\documentclass[11pt]{amsart}
\usepackage{a4wide}
\usepackage[T1]{fontenc}
\usepackage{amssymb,amsmath,amsthm,latexsym}
\usepackage{mathrsfs}
\usepackage[usenames,dvipsnames]{color}
\usepackage{euscript}
\usepackage{graphicx}
\usepackage{mdwlist}
\usepackage{enumerate}
\usepackage{mathtools,dsfont,wasysym}
\usepackage{bbm}
\usepackage{stmaryrd}
\usepackage{centernot}
\usepackage{todonotes}
\usepackage{hyperref}
\hypersetup{colorlinks=true,  linkcolor=black, citecolor=teal, urlcolor=cyan}

\makeatletter
\@namedef{subjclassname@2020}{\textup{2020} Mathematics Subject Classification}
\makeatother

\newtheorem{theorem}{Theorem}[section]
\newtheorem{lemma}[theorem]{Lemma}
\newtheorem{corollary}[theorem]{Corollary}
\newtheorem{proposition}[theorem]{Proposition}

\newcounter{maintheorem}

\theoremstyle{remark}
\newtheorem{remark}[theorem]{Remark}
\theoremstyle{definition}
\newtheorem{definition}[theorem]{Definition}

\newtheorem{example}[theorem]{Example}
\numberwithin{equation}{section}
\makeatother

\newcommand{\R}{\mathbb{R}}
\newcommand{\N}{\mathbb{N}}
\newcommand{\Z}{\mathbb{Z}}

\newcommand{\nn}[1]{{\left\vert\kern-0.25ex\left\vert\kern-0.25ex\left\vert #1
\right\vert\kern-0.25ex\right\vert\kern-0.25ex\right\vert}}

\renewcommand{\leq}{\leqslant}
\renewcommand{\geq}{\geqslant}
\newcommand{\dif}{\nabla_{xy}}

\DeclareMathOperator{\Vol}{Vol}

\newcounter{smallromans}

{\end{list}}

\newcommand{\bone}{\mathbbm{1}}

\begin{document}
\title[Nonexistence for the wave equation on graphs]{Nonexistence results for the semilinear\\ wave equation on graphs}

\author[D.D. Monticelli]{Dario D. Monticelli}
\address[D.D. Monticelli]{Politecnico di Milano, Dipartimento di Matematica, Piazza Leonardo da Vinci 32, 20133 Milano, Italy}
\email{dario.monticelli@polimi.it}

\author[F. Punzo]{Fabio Punzo}
\address[F. Punzo]{Politecnico di Milano, Dipartimento di Matematica, Piazza Leonardo da Vinci 32, 20133 Milano, Italy}
\email{fabio.punzo@polimi.it}

\author[J. Somaglia]{Jacopo Somaglia}
\address[J. Somaglia]{Politecnico di Milano, Dipartimento di Matematica, Piazza Leonardo da Vinci 32, 20133 Milano, Italy}
\email{jacopo.somaglia@polimi.it}

\keywords{Graphs, semilinear hyperbolic equations on graphs, nonexistence of global solutions, distance function, weighted volume, test functions.}
\subjclass[2020]{35A01, 35A02, 35B44, 35L05, 35L71, 35R02.}
\date{\today}

\begin{abstract}
We investigate the semilinear wave equation with a positive potential on weighted graphs. We establish sufficient conditions for the nonexistence of global-in-time solutions. Both nonnegative and sign-changing solutions are considered. In particular, the proof for sign-changing
solutions relies on a novel technique for this type of result.
\end{abstract}
\maketitle


\section{Introduction}
We study the nonexistence of global solutions to the Cauchy problem
for a semilinear wave equation of the form:
\begin{equation}\label{e: maineq}
    \begin{cases}
    u_{tt} - \Delta u\geq v|u|^\sigma \quad &\mbox{ in } V\times (0,\infty),\\
    u=u_0\quad &\mbox{ in } V\times\{0\},\\
    u_t=u_1\quad &\mbox{ in } V\times\{0\},
        \end{cases}
\end{equation}
posed on a weighted graph $(V, \omega, \mu)$, where $\Delta$ denotes
the graph Laplacian. The function $v : V\times (0,\infty) \to \mathbb{R}$, commonly
referred to as the \emph{potential}, is assumed to be positive.
The exponent satisfies $\sigma > 1$, and the functions
$u_0, u_1 : V \to \mathbb{R}$ represent the prescribed initial data.

In recent years, there has been growing interest in studying partial
differential equations on graphs, especially those involving infinite
or weighted configurations (cf. \cite{Grig1, KLW, Mu1}).

Although much of the focus has traditionally been on elliptic problems,
which have been thoroughly investigated (see, for instance,
\cite{BP1, BMP1, GLY, GLY2, HW, MP, MPS1}), a considerable
body of work has also emerged around parabolic equations, as
demonstrated in papers such as
\cite{BCG, BMP2, DvC, GT, HL, H, HKS, KR, LW, M, MP2, MPS2, PT1, PT2}. The wave equation has likewise attracted attention, as discussed in
\cite{AZ, AN, FT, HaHua, LX}.

\smallskip

In contrast to the parabolic case, the theory surrounding nonexistence results for hyperbolic
equations has
developed more slowly. To the best of our knowledge, the earliest
results about the nonexistence of global solutions for the wave
equation in $\mathbb{R}^n$ are found in \cite{Kato80}, with a
broader framework discussed in \cite{Sha85, MP01}. These results show that the Cauchy problem
\begin{equation} \label{eq:waveR}
    \begin{cases}
    u_{tt} - \Delta u \geq |u|^q & \text{in } \mathbb{R}^n \times (0, \infty), \\
    u = u_0 & \text{in } \mathbb{R}^n\times\{0\}, \\
    u_t = u_1 & \text{in } \mathbb{R}^n\times\{0\},
    \end{cases}
\end{equation}
admits no nontrivial solution, provided that
\begin{equation*}
    1 < q \leq \frac{n+1}{n-1}
    \quad \text{and} \quad
    \liminf_{R \to \infty} \int_{B_R(0)} u_1(x)\, dx \geq 0.
\end{equation*}
Also, in \cite{MPS20} (see also \cite{Ru}) the hyperbolic problem
\begin{equation} \label{eq:hyperbolicM}
    \begin{cases}
    u_{tt} - \Delta u \geq v |u|^q & \text{in } M \times (0, \infty), \\
    u = u_0 & \text{in } M\times\{0\}, \\
    u_t = u_1 & \text{in } M\times\{0\},
    \end{cases}
\end{equation}
is analyzed, where $M$ is a complete, noncompact Riemannian manifold of dimension
$n$, equipped with a metric tensor $g$, and $\Delta$ denotes the
Laplace--Beltrami operator on $M$. The exponent satisfies $q > 1$,
and the potential $v$ belongs to $L^1_{\mathrm{loc}}(M \times [0, \infty))$
and is strictly positive almost everywhere. A nonexistence result is proved for very weak solutions to this problem,
under the assumption of a suitable lower bound on the Ricci curvature
and a weighted volume growth condition on geodesic balls, where the
weight depends explicitly on the potential $v$.

\smallskip

The issue of nonexistence of solutions on graphs has been investigated for elliptic equations e.g. in \cite{GHS, MPS1}, and for parabolic equations e.g. in \cite{GMP, GSXX, LW, MPS2, Wu}. In particular, \cite{MPS2} studies
 the nonexistence of nonnegative, nontrivial global solutions
to a class of semilinear parabolic inequalities of the type:
\begin{equation} \label{eq:parabolic}
    u_t \geq  \Delta u + v u^\sigma
    \quad \text{in } V \times (0, \infty),
\end{equation}
where $\sigma > 1$, and $v$ is a given positive potential. Assuming an upper bound on the Laplacian of the distance function
and a suitable space-time weighted volume growth condition, it is established
that problem \eqref{eq:parabolic} admits no nontrivial
nonnegative global solutions.

\medskip

In this paper, we consider both the case of nonnegative solutions and that of general sign-changing solutions. It is important to distinguish between these two cases, as we shall explain shortly. Indeed, we prove that the only nonnegative global-in-time solution to the Cauchy problem is the identically zero solution, under suitable hypotheses on the initial data and specific weighted space-time volume growth conditions. These conditions depend on the graph, the potential and the exponent \(\sigma\). Moreover, we assume an upper bound on the Laplacian of the distance function from a fixed point on the graph.

In this setting, the proof resembles, though with notable differences, the parabolic case on graphs and the hyperbolic one on Riemannian manifolds. It relies on a priori estimates obtained by choosing suitable test functions. It is worth emphasizing that these test functions are compactly supported.

In contrast, in the case of general solutions that may change sign, this approach no longer applies. While the proof still relies on a priori estimates and the selection of test functions, compactly supported test functions are inadequate (see Proposition \ref{nosuppcomp} and the comments before that). Instead, we employ test functions supported on the entire graph, but with sufficiently fast decay at infinity. However, this requires us to strengthen our assumptions. In particular, we must assume a stronger condition on the Laplacian—specifically, an estimate on the modulus of the Laplacian of the distance function from a fixed point. Additionally, the weighted volume growth condition is replaced by a more stringent one, and we further assume that $u$ belongs to a suitable weighted $\ell^1$ space.

This line of reasoning is novel and represents a departure from existing approaches used for elliptic and parabolic equations on graphs, as well as from the techniques employed for hyperbolic equations on \(\mathbb{R}^n\) and on manifolds, as discussed in the aforementioned literature.

\medskip

The paper is organized as follows. In Section \ref{prel}, we recall some basic properties of graphs and introduce the assumptions used throughout the paper. The main results for infinite graphs are presented in Section \ref{mr}. Nonexistence results for nonnegative solutions are established in Section \ref{proofnonneg}, while those concerning sign-changing solutions are proved in Section \ref{proofsignc}. Several examples illustrating the applicability of our results are provided in Section \ref{ex}.
Finally, as a complement, we discuss some related results in the setting of finite graphs.

\textbf{Acknowledgements.} The authors are members of GNAMPA-INdAM and are partially supported by
GNAMPA projects 2024. Moreover, Dario D. Monticelli and F. Punzo acknowledge that this work is part of the PRIN ”Geometric-analytic methods for PDEs and applications", ref. 2022SLTHCE, financially supported by the EU, in the framework of the "Next Generation EU initiative".

\section{Preliminaries}\label{prel}

In what follows we recall the basic definitions and notations on graphs.
Let $V$ be a countable set and $\mu:V\to (0,\infty)$ be a given function called \textit{node measure}.
Furthermore, let
\begin{equation*}
\omega:V\times V\to [0,\infty)
\end{equation*}
be a symmetric function, called \textit{edge weight}, with zero diagonal and finite sum, i.e.
\begin{equation}\label{omega}
\begin{array}{ll}
\text{(i)}\,\, \displaystyle\omega_{xy}=\omega_{yx}&\text{for all}\,\,\, (x,y)\in V\times V;\\
\text{(ii)}\,\, \displaystyle\omega_{xx}=0 \quad\quad\quad\,\, &\text{for all}\,\,\, x\in V;\\
\text{(iii)}\,\, \displaystyle \sum_{y\in V} \omega_{xy}<\infty \quad &\text{for all}\,\,\, x\in V\,.
\end{array}
\end{equation}
Thus, we define a  \textit{weighted graph} as the triplet $(V,\omega,\mu)$. Observe that assumption (ii) means that the graph has no loops.
\smallskip

\noindent Let $x,y\in V$. We say that $x$ is {\it connected} to $y$ and we write $x\sim y$, whenever $\omega_{xy}>0$.
 If $x$ and $y$ are connected, we denote by $(x,y)$ the {\it edge} of the graph between the vertices $x$ and $y$. By $E$ we denote the set of all edges of the graph.
A collection of vertices $ \{x_k\}_{k=0}^n\subset V$ is a {\it path} consisting of $n+1$ nodes if $x_k\sim x_{k+1}$ for all $k=0, \ldots, n-1.$

A weighted graph $(V,\omega,\mu)$  is said to be
\begin{itemize}
\item[(i)] {\em locally finite} if each vertex $x\in V$ has only finitely many $y\in V$ such that $x\sim y$;
\item[(ii)] {\em connected} if, for any two distinct vertices $x,y\in V$, there exists a path joining $x$ to $y$;
\item[(iii)] {\em undirected} if its edges do not have an orientation.
\end{itemize}
For any element $x\in V$ we will call {\em degree of $x$} the cardinality of the set $$\{y\in V\colon y\sim x\}.$$

\noindent A {\it pseudo metric} on $V$ is a symmetric map, with zero diagonal, $d:V\times V\to [0, \infty)$, which also satisfies the triangle inequality
\begin{equation*}
  d(x,y)\leq d(x,z)+d(z,y)\quad \text{for all}\,\,\, x,y,z\in V.
\end{equation*}
In general, $d$ is not a metric, since there may be distinct points $x, y\in V$ such that $d(x,y)=0\,.$
Finally, we define the \textit{jump size} $j>0$ of a general pseudo metric $d$ as
\begin{equation}\label{e14f}
j:=\sup\{d(x,y) \,:\, x,y\in V, \omega(x,y)>0\}.
\end{equation}

Let $\Omega$ be a subset of $V$,  we denote by $\bone_\Omega$ the characteristic function of $\Omega$. The {\it volume} of $\Omega\subset V$ is defined by
\[\operatorname{Vol}(\Omega):=\sum_{x\in \Omega}\mu(x)\,.\]

Let $\mathfrak F$ denote the set of all functions $f: V\to \mathbb R$\,. For any $f\in \mathfrak F$ and for all $x,y\in V$, let us give the following
\begin{definition}
Let $(V, \omega,\mu)$ be a weighted graph. For any $f\in \mathfrak F$,
the {\em (weighted) Laplace operator} on $(V, \omega, \mu)$ is
\begin{equation*}
\Delta f(x):=\frac{1}{\mu(x)}\sum_{y\in V}\omega_{xy}[f(y)-f(x)]=\frac 1{\mu(x)}\sum_{y\sim x}\omega_{xy}\nabla_{xy}f\quad \text{ for all }\, x\in V\,,
\end{equation*}
\end{definition}
where
\[\nabla_{xy}f:= f(y)-f(x)\,.\]

It is straightforward to show, for any $f,g\in \mathfrak F$, the validity of
the {\it integration by parts formula}
\begin{equation}\label{e4f}
\sum_{x\in V}[\Delta f(x)] g(x) \mu(x)=-\frac 1 2\sum_{x,y\in V}\omega_{xy}(\dif f)(\dif g)= \sum_{x\in V}f(x) [\Delta g(x)] \mu(x)\,,
\end{equation}
provided that at least one of the functions $f, g\in \mathfrak F$ has {\it finite} support.

We recall the definition of solution in this context.
\begin{definition}\label{defn2}
    We say that $u\colon V\times [0,\infty)\to \R$ is a \textit{very weak solution} of \eqref{e: maineq} if for every fixed $x\in V$ one has $u(x,\cdot)\in L^1_{\textup{loc}}([0,\infty))\cap L_{\textup{loc}}^{\sigma}([0,\infty),v(x,t)dt),$  and
    \begin{equation}\label{e: veryweaksol}
    \begin{split}
    \int_0^{\infty}\sum_{x\in V} u(x,t)\varphi_{tt}(x,t)\mu(x) dt&- \int_0^{\infty}\sum_{x\in V} \Delta u(x,t) \varphi (x,t)\mu(x) dt + \sum_{x\in V} u_0(x) \varphi_t(x,0)\mu(x)\\ &- \sum_{x\in V}u_1(x) \varphi(x,0)\mu(x)\geq \int_{0}^{\infty}\sum_{x\in V} v(x,t)|u(x,t)|^\sigma\varphi(x,t)\mu(x)dt,
    \end{split}
    \end{equation}
       for every $\varphi:V\times[0,\infty)\rightarrow\R$ such that $\varphi\geq0$, $\operatorname{supp}\varphi\subset[0,T]\times A$ for some $T>0$ and some finite set $A\subset V$ and $\varphi(x,\cdot)\in C^{2}([0,\infty))$ for every $x\in V$.
\end{definition}

In this paper, we will need the following assumptions:
\begin{equation}\label{e7f}
	\begin{aligned}
		\text{(i)}\,\,\, & (V, \omega, \mu) \text{ is a connected, locally finite, undirected, weighted graph};\\
		\text{(ii)}\,\, \, &  \text{ there exists a constant } C>0 \text{ such that for every } x\in V, \\
		&\qquad \sum_{y\sim x}\omega_{xy}\leq C \mu(x);\\
		\text{(iii)} \,\,\,& \text{there exists a \textit{pseudo metric}}\,\, d \,\,\,\text{such that its jump size $j$ is finite}; \\
		\text{(iv)}\,\,\,& \text{the ball}\,\,\, B_r(x) \,\,\,\text{with respect to}\,\,\, d\,\,\, \text{is a finite set, for any}\,\,\, x\in V,\,\,\, r>0;\\
		\text{(v)}\,\, & \text{ for some } x_0\in V,\, R_0>1,\, \alpha\in [0, 1],\, C>0 \text{ there holds } \\
		& \qquad\Delta d(x, x_0)\leq \frac C{d^\alpha(x, x_0)} \text{ for any }\, x\in V\setminus B_{R_0}(x_0).
        \end{aligned}
\end{equation}
For some of our results we will need to strengthen condition \eqref{e7f} requiring
\begin{equation}\label{e7fn}
\begin{aligned}
&\text{(i)-(iv)} \text{ of condition } \eqref{e7f};\\
&\text{(v$'$)}\,\,  \text{ for some } x_0\in V,\, R_0>1,\, \alpha\in [0, 1],\, C>0 \text{ there holds } \\
		& \qquad|\Delta d(x, x_0)|\leq \frac C{d^\alpha(x, x_0)} \text{ for any }\, x\in V\setminus B_{R_0}(x_0).
        \end{aligned}
\end{equation}
\begin{remark}\label{r: carattmodulolaplaciano}
    Let $\alpha\in [0,1]$, and $R_0\geq 2j$. Then
    \begin{equation}\label{e: stima laplaciano distanza modulo}
        |\Delta d(x,x_0)|\leq \frac{C}{d^\alpha(x,x_0)} \,\, \mbox{ for any } x\in V\setminus B_{R_0}(x_0),
    \end{equation}
    if and only if
    \begin{equation}\label{e: stima laplaciano distanza modulo con esponente}
        |\Delta [d^{1+\alpha}(x,x_0)]|\leq C \,\,\, \mbox{ for any } x\in V.
    \end{equation}
    Indeed, using Taylor expansions of the function $p\mapsto p^{1+\alpha}$ we have, for every $p\geq 0, r\geq 0$,
    \begin{equation*}
        p^{1+\alpha}-r^{1+\alpha}=(1+\alpha)r^{\alpha}(p-r) + \frac{\alpha(1+\alpha)}{2}\eta^{\alpha-1} (p-r)^2,
    \end{equation*}
    for some $\eta$ between $p$ and $r$. Let $x\in V\setminus B_{R_0}(x_0)$. By taking $p=d(y,x_0)$, $r=d(x,x_0)$, we get
    \begin{equation}\label{e: espansione di Taylor}
    \begin{split}
        \Delta [d^{1+\alpha}(x,x_0)]&=\frac{1}{\mu(x)}\sum_{y\sim x}\omega_{xy}(d^{1+\alpha}(y,x_0)-d^{1+\alpha}(x,x_0))\\
        &=\frac{1}{\mu(x)}\sum_{y\sim x}(1+\alpha)\omega_{xy}d^{\alpha}(x,x_0)(d(y,x_0)-d(x,x_0))\\
        &\,\,\,\,\,+\frac{1}{\mu(x)}\sum_{y\sim x}\frac{\alpha(1+\alpha)}{2}\omega_{xy}\eta^{\alpha-1}(d(y,x_0)-d(x,x_0))^2,
   \end{split}
    \end{equation}
    for some $\eta$  between $d(y,x_0)$ and $d(x, x_0)$. Suppose that \eqref{e: stima laplaciano distanza modulo} holds, then from \eqref{e: espansione di Taylor}, since $\alpha-1\leq 0$ and $\eta\geq \min\{d(y,x_0), d(x, x_0)\}\geq d(x,x_0)-j$, we get
\begin{equation*}
    \begin{split}
        |\Delta [d^{1+\alpha}(x,x_0)]|\leq (1+\alpha) d^\alpha(x,x_0)|\Delta d(x,x_0)| + \frac{\alpha(\alpha+1)}{2\mu(x)}\sum_{y\sim x}\omega_{xy} (d(x,x_0)-j)^{\alpha-1} j^2\leq C,
    \end{split}
\end{equation*}
for every $x\in V\setminus B_{R_0}(x_0)$. For every $x\in B_{R_0}(x_0)$, by finiteness of $B_{R_0}(x_0)$, for a possibly larger constant $C>0$ we have that $|\Delta [d^{1+\alpha}(x,x_0)]|\leq C$ . Then $$|\Delta [d^{1+\alpha}(x,x_0)]|\leq C \quad \text{ for every } x\in V.$$\\
Now assume \eqref{e: stima laplaciano distanza modulo con esponente}. Then, for each $x\in V$ such that $d(x,x_0)>R_0$ we have by \eqref{e: espansione di Taylor}
\begin{equation*}
    |d^\alpha(x,x_0)\Delta  d(x,x_0)|\leq \frac{|\Delta [d^{1+\alpha}(x,x_0)]|}{(1+\alpha)} + \frac{\alpha}{2\mu(x)}\sum_{y\sim x}\omega_{xy}(d(x,x_0) - j)^{\alpha-1}j^2\leq C,
\end{equation*}
which implies $|\Delta d(x,x_0)|\leq \frac{C}{d^\alpha(x,x_0)}$ for every $x\in V\setminus B_{R_0}(x_0)$.

\smallskip
Similarly, (see \cite[Remark 2.5]{MPS2}) one can prove that if $\alpha\in[0,1]$ then
\begin{equation*}
        \Delta d(x,x_0)\leq \frac{C}{d^\alpha(x,x_0)} \,\, \mbox{ for every } x\in V\setminus B_{R_0}(x_0),
    \end{equation*}
    if and only if
    \begin{equation*}
        \Delta [d^{1+\alpha}(x,x_0)]\leq C \,\,\, \mbox{ for every } x\in V.
    \end{equation*}
\end{remark}

\begin{remark}\label{rem3.999}
   Let $(V,\omega,\mu)$ be a  weighted graph satisfying \eqref{e7f} (i)-(iv). Then it holds that
    \begin{equation*}
        |\Delta d(x,x_0)|=\left|\frac{1}{\mu(x)}\sum_{y\sim x}\omega_{xy}(d(y,x_0)-d(x,x_0))\right|\leq  \frac{1}{\mu(x)} \sum_{y\sim x}\omega_{xy}|d(x,x_0)-d(y,x_0)|\leq jC.
    \end{equation*}
   In particular, $(V,\omega,\mu)$  satisfies \eqref{e7f}-(v) and \eqref{e7fn}-(v') for $\alpha=0$.
\end{remark}

\begin{remark}
Let $d^*$ be the {\em combinatorial graph distance} on $V$, that is the distance which, for any two
vertices $x,y\in V$, counts the least number of edges in a path connecting $x$ and $y$, and let
\[\mu^*(x):=\sum_{y\in V}\omega(x,y), \quad x\in V\,.\]
If $(V, \omega, \mu^*)$ is a connected, locally finite, undirected weighted graph, then $d^*$ has jump size $j^*=1$, and conditions \eqref{e7f} and \eqref{e7fn} are satisfied by $d^*$ with $\alpha=0$, by Remark \ref{rem3.999}.
\end{remark}

\section{Main results}\label{mr}
\subsection{Nonexistence for nonnegative global solutions}
We show that, under suitable assumptions, the only global nonnegative solution of \eqref{e: maineq} is the zero solution.

\begin{theorem}\label{teo1}
	Let assumption \eqref{e7f} be satisfied. Assume the graph $(V,\omega,\mu)$ is infinite.  Let $v\colon V\times [0,\infty)\to \R$ be a positive function, $\sigma>1$ and $\theta_1\geq 2$, $\theta_2\geq 2$ be such that $\frac{\theta_1}{\theta_2}\geq \frac{1 + \alpha}{2}$, with $\alpha\in[0,1]$ as in \eqref{e7f}. Suppose that \begin{enumerate}
    \item for every $R\geq R_0>1$
	\begin{equation}\label{e: stimavolumi}
		\int_0^{\infty}\sum_{x\in V} \bone_{E_R}(x,t) v^{-\frac{1}{\sigma-1}}(x,t)\mu(x)dt\leq CR^{\frac{(1+\alpha)\sigma}{\sigma-1}},
	\end{equation}
where
\begin{equation}\label{22}
    E_R\coloneqq\{(x,t)\in V\times [0,\infty)\colon R^{\theta_1}\leq d^{\theta_1}(x,x_0)+t^{\theta_2}\leq 2 R^{\theta_1}\}
\end{equation}
and  $x_0\in V$ as in \eqref{e7f};
\item \begin{equation}\label{e: liminf}
    \liminf_{R\to\infty}\left\{\sum_{x\in B_R(x_0)}u^+_1(x)\mu(x) - \sum_{x\in B_{2R}(x_0)}u_1^-(x)\mu(x)\right\}\geq 0.
\end{equation}
\end{enumerate}
Let $u\colon V\times [0,\infty)\to \R$ be a nonnegative very weak solution of \eqref{e: maineq}, then $u\equiv 0$.
\end{theorem}

\begin{corollary}\label{c: segno}
    	Let assumption \eqref{e7f} be satisfied. Assume the graph $(V,\omega,\mu)$ is infinite.  Let  $\sigma>1$,  $v\colon V\times [0,\infty)\to \R$ be a positive function such that $v(x,t)\geq g(x)$ for every $(x,t)\in V\times [0,\infty)$, for some positive function $g\colon V\to \R$. Suppose that $u$ is a nonnegative very weak solution of \eqref{e: maineq}. Suppose that
        \begin{enumerate}
        \item \begin{equation}\label{e50}
            \sum_{x\in B_R(x_0)}\mu(x)g^{-\frac{1}{\sigma-1}}(x)\leq CR^\frac{(1+\alpha)(\sigma+1)}{2(\sigma-1)}
        \end{equation}
            \item \begin{equation}
    \liminf_{R\to\infty}\left\{\sum_{x\in B_R(x_0)}u^+_1(x)\mu(x) - \sum_{x\in B_{2R}(x_0)}u_1^-(x)\mu(x)\right\}\geq 0.
    \end{equation}
        \end{enumerate}
        Then, $u\equiv 0$.
\end{corollary}

The following corollary follows immediately from the previous result.

\begin{corollary}\label{c: segno2}
    	Let assumption \eqref{e7f} be satisfied. Assume the graph $(V,\omega,\mu)$ is infinite. Let $\sigma>1$ and assume that for every $R\geq R_0>1$
    \begin{equation}\label{e: stimavolumi-C3}
		\operatorname{Vol}(B_R(x_0))\leq CR^{\frac{(1+\alpha)(\sigma+1)}{2(\sigma-1)}},
	\end{equation}
with $\alpha, x_0$ as in \eqref{e7f}. Let $u\colon V\times [0,\infty)\to \mathbb{R}$ be a nonnegative very weak solution of \eqref{e: maineq} with $v\equiv1$. If
\begin{equation}
    \liminf_{R\to\infty}\left\{\sum_{x\in B_R(x_0)}u^+_1(x)\mu(x) - \sum_{x\in B_{2R}(x_0)}u_1^-(x)\mu(x)\right\}\geq 0,
    \end{equation}
    then $u\equiv 0$.
\end{corollary}

\subsection{Nonexistence for sign-changing global solutions}

In order to study the problem for sign-changing solutions we have to introduce
the following space
\[
X_\delta
:=
\left\{
f:V\to\mathbb R:
\sum_{x\in V}|f(x)|e^{-\delta d(x,x_0)}\mu(x)<\infty
\right\},
\]
for $\delta>0$. Note that $X_\delta$ can be regarded as a weighted
$\ell^1$ space. We are now in the position of stating the second main result
of our paper, in which we make no assumptions on the sign of the solution
$u$.

\begin{theorem}\label{thm:sign-changing}
Let assumption \eqref{e7fn} be satisfied. Assume the graph
$(V,\omega,\mu)$ is infinite. Let
$v:V\times[0,\infty)\to\mathbb R$ be a positive function, $\sigma>1$,
and let $\delta_0,\delta_1>0$.
Suppose that $u$ is a very weak solution of \eqref{e: maineq} such that
\[
u\in L^1_{\mathrm{loc}}([0,\infty);X_{\delta_0}).
\]
Assume moreover that
\begin{equation}\label{eq:initial-mass-sign}
u_1^-\in\ell^1(V,\mu),\qquad \qquad
\sum_{x\in V}u_1(x)\mu(x)\geq0.
\end{equation}
Finally, suppose that for every $R\geq R_0>1$ it holds that
\begin{equation}\label{eq:sign-changing-volume-1}
\int_{R^{\frac{1+\alpha}{2}}}^{2R^{\frac{1+\alpha}{2}}}
\sum_{x\in B_R(x_0)}
v^{-\frac1{\sigma-1}}(x,t)
e^{-\delta_1\frac{d(x,x_0)}{R}}
\mu(x)\,dt
\leq
CR^{\frac{(\alpha+1)\sigma}{\sigma-1}},
\end{equation}
and
\begin{equation}\label{eq:sign-changing-volume-2}
\int_0^{2R^{\frac{1+\alpha}{2}}}
\sum_{x\in V\setminus B_R(x_0)}
v^{-\frac1{\sigma-1}}(x,t)
e^{-\delta_1\frac{d(x,x_0)}{R}}
\mu(x)\,dt
\leq
CR^{\frac{(\alpha+1)\sigma}{\sigma-1}}.
\end{equation}
Then $u\equiv0$.
\end{theorem}

\begin{corollary}\label{c: nosegno}
Let assumption \eqref{e7fn}  be satisfied. Assume the graph
$(V,\omega,\mu)$ is infinite. Let $\sigma>1$ and let
$v:V\times[0,\infty)\to\mathbb R$ be a positive function such that
\[
v(x,t)\geq g(x)
\qquad\text{for every }(x,t)\in V\times[0,\infty),
\]
for some positive function $g:V\to\mathbb R$.
Suppose that $u$ is a very weak solution of \eqref{e: maineq} such that, for
some $\delta_0>0$,
\[
u\in L^1_{\mathrm{loc}}([0,\infty);X_{\delta_0}),
\]
and assume that
\[
u_1^-\in\ell^1(V,\mu),
\qquad
\sum_{x\in V}u_1(x)\mu(x)\geq0.
\]
Moreover, suppose that for some $\delta_1>0$ and every $R\geq R_0>1$,
\begin{equation}\label{eq:sign-changing-g}
\sum_{x\in V}
g^{-\frac1{\sigma-1}}(x)
e^{-\delta_1\frac{d(x,x_0)}{R}}
\mu(x)
\leq
CR^{\frac{(\alpha+1)(\sigma+1)}{2(\sigma-1)}}.
\end{equation}
Then $u\equiv0$.
\end{corollary}

\begin{corollary}\label{c: nosegnoV=1}
Let assumption \eqref{e7fn} be satisfied. Assume the graph
$(V,\omega,\mu)$ is infinite. Let $\sigma>1$ and suppose that $u$ is a
very weak solution of \eqref{e: maineq} with $v\equiv1$ such that, for some
$\delta_0>0$,
\[
u\in L^1_{\mathrm{loc}}([0,\infty);X_{\delta_0}).
\]
Assume moreover that
\[
u_1^-\in\ell^1(V,\mu),
\qquad
\sum_{x\in V}u_1(x)\mu(x)\geq0.
\]
Suppose that
\begin{equation}\label{eq:annular-volume}
\operatorname{Vol}\bigl(B_{R+1}(x_0)\setminus B_R(x_0)\bigr)
\leq
CR^{\frac{(\alpha+1)(\sigma+1)}{2(\sigma-1)}-1},
\end{equation}
for every $R\geq R_0>1$. Then $u\equiv0$.
\end{corollary}

\section{Proof of Theorem \ref{teo1}}\label{proofnonneg}
In this section we prove the hyperbolic counterpart of \cite[Theorem 2.6]{MPS2}. Let us mention that we will skip the overlapping parts and refer to \cite{MPS2} for all the details.

\begin{proof}[Proof of Theorem \ref{teo1}]\label{proofs}
Let $\varphi\in C^2([0,\infty))$ be a cut-off function such that $\varphi\equiv 1$ in $[0,1]$, $\varphi\equiv 0$ in $[2,\infty)$, and $\varphi'\leq 0$. For each $x\in V$, $t\in [0,\infty)$ and $R\geq R_0$ we define
\[
\psi_R(x,t)=\frac{t^{\theta_2}+ d(x,x_0)^{\theta_1}}{R^{\theta_1}},\qquad\phi_R(x,t)=\varphi(\psi_R(x,t)).
\]

We are going to prove the following upper bounds
\begin{enumerate}[(i)]
	\item\label{i: stima laplaciano}  there exists $C\geq 0$ such that for every $x\in V$ and $t\in [0,\infty)$ it holds  \[-\Delta\phi_R(x,t)\leq \frac{C}{R^{1+\alpha}} \bone_{F_R}(x,t),\]
 where $F_R\coloneqq\{(x,t)\in V\times [0,\infty)\colon (R/2)^{\theta_1}\leq d(x,x_0)^{\theta_1}+t^{\theta_2}\leq  (4R)^{\theta_1}\}$;
	\item\label{i: stima derivata in tempo} there exists $C\geq 0$ such that for every $x\in V$ and $t\in [0,\infty)$ it holds  \[\left|\frac{\partial \phi_R}{\partial t}(x,t)\right|\leq \frac{C}{R^{\frac{\theta_1}{\theta_2}}}\bone_{E_R}(x,t).\]
    \item \label{i: stima derivata seconda in tempo} there exists $C\geq 0$ such that for every $x\in V$ and $t\in [0,\infty)$ it holds  \[\frac{\partial^2 \phi_R}{\partial t^2}(x,t)\leq \frac{C}{R^{2\frac{\theta_1}{\theta_2}}}\bone_{E_R}(x,t).\]
\end{enumerate}

We prove only \eqref{i: stima derivata in tempo} and \eqref{i: stima derivata seconda in tempo}, we refer to \cite[Section 3]{MPS2} for the proof of \eqref{i: stima laplaciano}. Let $x\in V$ and $t\geq 0$, we get

\begin{equation*}
	\begin{split}
		\left|\frac{\partial \phi_R}{\partial t}(x,t)\right|=|\varphi'(\psi_R(x,t))|\theta_2 \frac{t^{\theta_2-1}}{R^{\theta_1}} \leq \frac{C}{R^{\theta_1}}R^{\theta_1-\frac{\theta_1}{\theta_2}}\bone_{E_R}(x,t)=\frac{C}{R^\frac{\theta_1}{\theta_2}}\bone_{E_R}(x,t),
	\end{split}
\end{equation*}
and
\begin{equation*}
	\begin{split}
		\frac{\partial^2 \phi_R}{\partial t^2}(x,t)&=\varphi''(\psi_R(x,t))\theta_2^2 \frac{t^{2\theta_2-2}}{R^{2\theta_1}} +\varphi'(\psi_R(x,t))\theta_2(\theta_2 - 1)\frac{t^{\theta_2-2}}{R^{\theta_1}} \\
        &\leq C\left(\frac{R^{\frac{\theta_1}{\theta_2}(2\theta_2 -2)}}{R^{2\theta_1}} + \frac{R^{\frac{\theta_1}{\theta_2}(\theta_2 -2)}}{R^{\theta_1}} \right)\bone_{E_R}(x,t)\\
        &\leq\frac{C}{R^{2\frac{\theta_1}{\theta_2}}}\bone_{E_R}(x,t),
	\end{split}
\end{equation*}
where we have used that $\varphi'(\psi_R(x,t))\equiv0$ and $\varphi''(\psi_R(x,t))\equiv0$ on $E_R^c$, that $\varphi'$ and $\varphi''$ are bounded, and that if $(x,t)\in E_R$, it holds that $t\leq CR^\frac{\theta_1}{\theta_2}$.

The next step is to prove that there exists $C>0$ such that
\[
\int_0^{\infty}\sum_{x	\in V}\mu(x)u^{\sigma}(x,t)v(x,t) \, dt\leq C.
\]
In order to do that we first observe that the support of the test function $\phi_R(x,t)$ is contained in the subset $$Q_R\coloneqq B_{2^\frac{1}{\theta_1}R}(x_0)\times \big[0,2^{\frac{1}{\theta_2}}R^{\frac{\theta_1}{\theta_2}}\big],$$ so that all sums in $x\in V$ are finite, while all the integrals in the time-variable are on compact domains. Moreover $0\leq\phi_R\leq1$ in $V\times[0,\infty)$ and $\phi_R(x,\cdot)\in C^2([0,\infty))$ for all $x\in V$. Since $u$ is a very weak solution of \eqref{e: maineq}, by \eqref{e: veryweaksol} testing the equation with $\phi_R^s$ and  $s>\frac{2\sigma}{\sigma-1}$, we have
\begin{equation}\label{3}
\begin{aligned}
\int_0^{\infty}\sum_{x\in V} \mu(x)v(x,t)u^\sigma(x,t)\phi^s_R(x,t)\, dt
&\leq - \int_0^{\infty}\sum_{x\in V}\Delta u(x,t)\phi_R^s(x,t)\mu(x)\, dt\\
&\,\,\,\,\,\,\,\,+\sum_{x\in V}u_0(x)(\phi_R^s)_t(x,0)\mu(x)-\sum_{x\in V}u_1(x)\phi_R^s(x,0)\mu(x)\\
&\,\,\,\,\,\,\,\, +\int_{0}^{\infty}\sum_{x\in V}u(x,t)(\phi_R^s)_{tt}(x,t)\mu(x)dt.
\end{aligned}
\end{equation}
Now we are going to estimate each term one by one. In the same way as in \cite[Theorem 2.6]{MPS2}, from claim (i) and by a convexity argument, we get
\begin{equation*}
\begin{split}
    -\int_0^{\infty}\sum_{x\in V}\mu(x)\Delta u(x,t)\phi^s_R(x,t)\, dt&\leq \frac{C}{R^{1+\alpha}} \int_0^{\infty}\sum_{x\in V}\mu(x)u(x,t)\phi_R^{s-1}(x,t)\bone_{F_R}(x,t) \,dt.
\end{split}
\end{equation*}
Observing that $(\phi_R^s)_t(x,0)\equiv 0$, we get
\begin{equation*}
    \sum_{x\in V}u_0(x)(\phi_R^s)_t(x,0)\mu(x)\equiv 0.
\end{equation*}
Moreover
\begin{equation*}
\begin{split}
    -\sum_{x\in V}u_1(x)\phi_R^s(x,0)\mu(x)&= -\sum_{x\in V} u_1(x) \varphi^s\left(\frac{d(x_0,x)^{\theta_1}}{R^{\theta_1}}\right)\mu(x)\\
    &=-\sum_{x\in B_{2R}(x_0)}u_1(x)\varphi^s\left(\frac{d(x_0,x)^{\theta_1}}{R^{\theta_1}}\right)\mu(x),
\end{split}
\end{equation*}
where in the last equality we have used the fact that $\varphi\equiv 0$ in $[2,\infty)$. We get
\begin{equation*}
\begin{split}
    &-\sum_{x\in B_{2R}(x_0)}u_1(x)\varphi^s\left(\frac{d(x_0,x)^{\theta_1}}{R^{\theta_1}}\right)\mu(x)\\
    &\qquad=-\sum_{x\in B_{2R}(x_0)}u^+_1(x)\varphi^s\left(\frac{d(x_0,x)^{\theta_1}}{R^{\theta_1}}\right)\mu(x) +\sum_{x\in B_{2R}(x_0)}u_1^-(x)\varphi^s\left(\frac{d(x_0,x)^{\theta_1}}{R^{\theta_1}}\right)\mu(x)\\
    &\qquad\leq -\sum_{x\in B_{R}(x_0)}u_1^+(x)\mu(x) + \sum_{x\in B_{2R}(x_0)} u_1^-(x)\mu(x),
\end{split}
\end{equation*}
where in the first inequality we have used that $0\leq \varphi \leq 1$ on $[0,\infty)$ and $\varphi\equiv 1$ on $[0,1]$. It remains to estimate the last term
\begin{equation*}
\begin{split}
\int_0^{\infty} \sum_{x\in V}u(x,t)(\phi_R^s)_{tt}(x,t)\mu(x)\, dt&=\int_0^{\infty} \sum_{x\in V}u(x,t)s(s-1)\phi_R^{s-2}(x,t)\left[\frac{\partial \phi_R}{\partial t}(x,t)\right]^2\mu(x)\, dt\\
&\,\,\,\, \,\,\,+ \int_0^{\infty}\sum_{x\in V}u(x,t)s\phi_R^{s-1}(x,t)\frac{\partial^2\phi_R}{\partial t^2}(x,t)\mu(x)\, dt\\
&\leq \frac{C}{R^{2\frac{\theta_1}{\theta_2}}}\int_0^{\infty}\sum_{x\in V}u(x,t) \phi^{s-2}_R(x,t)\mu(x)\bone_{E_R}(x,t) \, dt\\
&\leq \frac{C}{R^{1+\alpha}}\int_0^{\infty}\sum_{x\in V}u(x,t) \phi^{s-2}_R(x,t)\mu(x)\bone_{E_R}(x,t) \, dt,
\end{split}
\end{equation*}
where we have used \eqref{i: stima derivata in tempo}, \eqref{i: stima derivata seconda in tempo}, and $\frac{\theta_1}{\theta_2}\geq \frac{1+\alpha}{2}$.
Combining everything and observing that $E_R\subset F_R$, from \eqref{3}, we get
\begin{equation*}
\begin{split}
\int_0^{\infty}\sum_{x\in V} \mu(x)v(x,t)u^\sigma(x,t)\phi^s_R(x,t)\, dt&\leq \frac{C}{R^{1+\alpha}}\int_0^{\infty}\sum_{x\in V}u(x,t) \phi^{s-2}_R(x,t)\mu(x)\bone_{F_R}(x,t)\, dt\\
&\,\,\,\, \,\,\,- \left(\sum _{x\in B_{R}(x_0)}u_1^+(x)\mu(x)-\sum _{x\in B_{2R}(x_0)}u_1^-(x)\mu(x)\right).
\end{split}
\end{equation*}
By applying Young's inequality we get
\begin{equation*}
	\begin{aligned}
& \int_0^{\infty}\sum_{x\in V} \mu(x)v(x,t)u^{\sigma}(x,t)\phi^s_R(x,t)\, dt\\
&\quad \leq \frac{1}{\sigma}	\int_0^{\infty}\sum_{x\in V} \mu(x)v(x,t)u^{\sigma}(x,t)\phi^{s}_R(x,t)\bone_{F_R}(x,t)\,dt\\
 &\qquad+ \frac{C}{R^{\frac{(1+\alpha)\sigma}{\sigma-1}}}\int_0^{\infty}\sum_{x\in V}\mu(x)\phi_R^{s-\frac{2\sigma}{\sigma-1}}(x,t)v^{-\frac{1}{\sigma-1}}(x,t)\bone_{F_R}(x,t)\,dt\\
 &\qquad- \left(\sum _{x\in B_{R}(x_0)}u_1^+(x)\mu(x)-\sum _{x\in B_{2R}(x_0)}u_1^-(x)\mu(x)\right)\\
 &\quad\leq \frac{1}{\sigma}	\int_0^{\infty}\sum_{x\in V} \mu(x)v(x,t)u^{\sigma}(x,t)\phi^{s}_R(x,t)\,dt\\
 &\qquad + \frac{C}{R^{\frac{(1+\alpha)\sigma}{\sigma-1}}}	\int_0^{\infty}\sum_{x\in V}\mu(x)v^{-\frac{1}{\sigma-1}}(x,t)\bone_{F_R}(x,t)\,dt\\
 &\qquad - \left(\sum _{x\in B_{R}(x_0)}u_1^+(x)\mu(x)-\sum _{x\in B_{2R}(x_0)}u_1^-(x)\mu(x)\right).
	\end{aligned}
\end{equation*}
Thus, we have
\begin{equation*}
\begin{split}
	\begin{split}
		\int_0^{\infty}\sum_{x\in V} \mu(x)v(x,t)u^{\sigma}(x,t)\phi^s_R(x,t)\, dt&\leq \frac{C}{R^{\frac{(1+\alpha)\sigma}{\sigma-1}}}	\int_0^{\infty}\sum_{x\in V}\mu(x)v^{-\frac{1}{\sigma-1}}(x,t)\bone_{F_R}(x,t)\, dt\\&\,\,\,\, \,\,\,- \left(\sum _{x\in B_{R}(x_0)}u_1^+(x)\mu(x)-\sum _{x\in B_{2R}(x_0)}u_1^-(x)\mu(x)\right).
\end{split}
	\end{split}
\end{equation*}
By \eqref{e: stimavolumi} and following the same procedure of \cite[Theorem 2.6]{MPS2}, we get
\begin{equation*}
    \int_0^{\infty}\sum_{x\in V} \mu(x)v(x,t) u^{\sigma}(x,t)\phi_R^s(x,t)  \, dt \leq C - \left(\sum _{x\in B_{R}(x_0)}u_1^+(x)\mu(x)-\sum _{x\in B_{2R}(x_0)}u_1^-(x)\mu(x)\right).
\end{equation*}
Taking the $\liminf$ as $R\to \infty$, by \eqref{e: liminf}, we get
\begin{equation}
   \int_0^{\infty}\sum_{x\in V} \mu(x)v(x,t) u^{\sigma}(x,t) \, dt \leq C.
\end{equation}
The conclusion now follows, by an application of H\"older's inequality, in a similar way as \cite[Theorem 2.6]{MPS2}. This passage is standard and is not affected by the type of equation (parabolic or hyperbolic) that one is considering.
\end{proof}

\begin{proof}[Proof of Corollary \ref{c: segno}]
    We first fix $\theta_1= 2(1+\alpha)$ and $\theta_2=4$. Thus it holds that $\theta_1,\theta_2\geq 2$ and $\frac{\theta_1}{\theta_2}=\frac{1+\alpha}{2}$. Now we claim that condition \eqref{e: stimavolumi} of Theorem \ref{teo1} holds. Indeed, let $R\geq R_0>1$ and $E_R$ be defined as in the statement of Theorem \ref{teo1}. We have
    \begin{equation*}
        \begin{split}
\int_0^{\infty}\sum_{x\in V}\bone_{E_R}(x,t) v(x,t)^{-\frac{1}{\sigma-1}}\mu(x)\, dt&\leq \int_0^{\infty}\sum_{x\in V}\bone_{E_R}(x,t)g^{-\frac{1}{\sigma-1}}(x)\mu(x)\, dt\\
            & \leq \int_{0}^{2^{\frac{1}{\theta_2}}R^{\frac{1+\alpha}{2}}}\sum_{x\in V}g^{-\frac{1}{\sigma-1}}(x)\bone_{B_{2^{\frac{1}{\theta_1}}R}(x_0)}(x,t)\mu(x)\, dt\\
            &\leq CR^{\frac{(1+\alpha)\sigma}{\sigma-1}},
        \end{split}
    \end{equation*}
    where we have used that $E_R\subset B_{2^{{\frac{1}{\theta_1}}}R}(x_0)\times [0,2^{\frac{1}{\theta_2}}R^{\frac{1+\alpha}{2}}]$. This shows that $u$ must be equal to zero, by Theorem \ref{teo1}.
\end{proof}

\begin{proof}[Proof of Corollary \ref{c: segno2}]
This is a simple consequence of Corollary \ref{c: segno}, with $g\equiv 1$. Thus condition \eqref{e50} becomes \eqref{e: stimavolumi-C3}, and the result follows.
\end{proof}

\section{Proof of Theorem~\ref{thm:sign-changing}}\label{proofsignc}

A key technical tool of the proof of Theorem \ref{teo1} is the construction of a suitable compactly supported nonnegative cut-off function $\phi$ satisfying, for some $C>0$,
\[-\Delta \phi^s \leq - C \phi^{s-1}\Delta\phi \quad \text{ in } V,\]
for $s>1$ large enough. In order to prove nonexistence for sign-changing solutions, we would need to construct a suitable compactly supported nonnegative cut-off function $\phi$ satisfying
\[|\Delta \phi^s| \leq  C \phi^{s-1}|\Delta\phi| \quad \text{ in } V.\] The following
proposition shows that this is not possible. Thus, we need a different
approach in order to deal with sign-changing solutions.

\begin{proposition}\label{nosuppcomp}
Let $(V,\omega,\mu)$ be an infinite, connected, locally finite, weighted
graph. Let $\psi:V\to\mathbb R$. Suppose that $\psi\geq0$ and there exists
$x_1\in V$ such that $\psi(x_1)=0$. If there are constants
$C,\beta,\gamma>0$ and $F:V\to[0,\infty)$ such that
\[
|\Delta\psi^\gamma(x)|
\leq C\psi^\beta(x)F(x),
\]
for every $x\in V$, then $\psi\equiv0$ in $V$.
\end{proposition}

\begin{proof}
Let $x_1\in V$ be such that $\psi(x_1)=0$. We have
\[
0=C\psi^\beta(x_1)F(x_1)\mu(x_1)
\geq
\left|
\sum_{y\sim x_1}
\omega_{x_1y}
\bigl(\psi^\gamma(y)-\psi^\gamma(x_1)\bigr)
\right|
=
\sum_{y\sim x_1}\omega_{x_1y}\psi^\gamma(y)
\geq0.
\]
Hence $\psi(y)=0$ for every $y\sim x_1$. Since $V$ is connected, by
repeating the same argument we obtain $\psi\equiv0$ in $V$.
\end{proof}

In particular, a nontrivial nonnegative compactly supported function cannot
satisfy an estimate of the above type: at a vertex in its zero set adjacent
to its positive set, the graph Laplacian detects a positive neighboring
value. This is a specifically discrete obstruction and explains why, in the
sign-changing case, compactly supported spatial cut-offs cannot be used in
the same way as in the nonnegative case.

We next recall an integration by parts formula for exponentially decaying
functions.

\begin{proposition}\label{p: ibp}
    Let $(V,\omega,\mu)$ be an infinite graph satisfying \eqref{e7f} \textup{(i)-(iv)}. Let $u\colon V\to \R$ and $\varphi\colon V\to \R$ be such that, for some $x_0\in V$ and $\delta>0$, the following are satisfied
\begin{enumerate}[(i)]
    \item $\displaystyle\sum_{x\in V}\mu(x)|u(x)|e^{-\delta d(x,x_0)}<\infty$;
    \item $|\varphi(x)|\leq C e^{-\delta d(x,x_0)}$, for every $x\in V$.
\end{enumerate}
Then 
\[
\sum_{x\in V}\mu(x) u(x)\Delta \varphi(x)=\sum_{x\in V} \mu(x) \Delta u(x) \varphi(x),
\]
and both series converge absolutely.
\end{proposition}

\begin{proof}
We first show that
\begin{equation*}
    \sum_{x\in V}\mu(x)|\Delta u(x)||\varphi(x)|
\end{equation*}
converges. Indeed,
\begin{equation*}
\begin{split}
    \sum_{x\in V}\mu(x)|\Delta u(x)||\varphi(x)|&\leq \sum_{x\in V}\sum_{y\sim x}\omega_{xy}(|u(x)|+|u(y)|)|\varphi(x)|\\
    &=\sum_{x\in V}\sum_{y\sim x}\omega_{xy}|u(x)||\varphi(x)|+\sum_{x\in V}\sum_{y\sim x}\omega_{xy}|u(y)||\varphi(x)|.
\end{split}
\end{equation*}
We observe that
\begin{equation*}
\begin{split}
    \sum_{x\in V}\sum_{y\sim x}\omega_{xy}|u(x)||\varphi(x)|&\leq C\sum_{x\in V} \mu(x)|u(x)||\varphi(x)|\\
    &\leq C\sum_{x\in V}\mu(x)|u(x)|e^{-\delta d(x,x_0)}
\end{split}
\end{equation*}
and that, since $d(y,x_0)\leq  j + d(x,x_0)$, we have
\begin{equation*}
\begin{split}
\sum_{x\in V}\sum_{y\sim x}\omega_{xy}|u(y)||\varphi(x)|&\leq C\sum_{x\in V}\sum_{y\sim x} \omega_{xy} |u(y)| e^{-\delta d(x,x_0)}\\
&\leq C e^{\delta j}\sum_{x\in V}\sum_{y\sim x} \omega_{xy}|u(y)|e^{-\delta d(y,x_0)}\\
&\leq C e^{\delta j}\sum_{y\in V}\mu(y)|u(y)| e^{-\delta d(y,x_0)}.
\end{split}
\end{equation*}
Thus, we have that
\begin{equation*}
    \sum_{x\in V}\mu(x)\varphi (x) \Delta u(x)
\end{equation*}
converges absolutely, and moreover
\begin{equation*}
     \sum_{x\in V}\sum_{y\sim x}\omega_{xy}(|u(x)|+|u(y)|)|\varphi(x)|<\infty.
\end{equation*}
Therefore, we obtain
\begin{equation*}
    \begin{aligned}
   & \sum_{x\in V} \Delta u(x) \varphi(x) \mu(x)=\sum_{x\in V}\sum_{y\sim x}\omega_{xy}(u(y)-u(x))\varphi(x)\\ &=\sum_{x\in V}\sum_{y\sim x}\omega_{xy} u(y)\varphi(x) - \sum_{x\in V}\sum_{y\sim x} \omega_{xy} u(x)\varphi(x)\\
        &=\sum_{y\in V}\sum_{x\in V}\omega_{xy}u(x) \varphi(y)-\sum_{x\in V}\sum_{y\in V}\omega_{xy}u(x)\varphi(x)\\
        &=\sum_{x\in V}\sum_{y\in V}\omega_{xy}u(x)(\varphi(y) - \varphi(x))=\sum_{x\in V}u(x) \Delta\varphi(x) \mu(x)\,.
    \end{aligned}
\end{equation*}
This concludes the proof.
\end{proof}

The next lemma is the key additional ingredient needed in the sign-changing
case. It shows that the weighted summability assumption on the solution
automatically improves to every exponential weight.

\begin{lemma}\label{lem:bootstrap-Xdelta}
Let the assumptions of Theorem~\ref{thm:sign-changing} concerning the graph
and the potential be satisfied, and assume only condition
\eqref{eq:sign-changing-volume-2}. Suppose that, for some $\beta>0$,
\[
u\in L^1_{\mathrm{loc}}([0,\infty);X_\beta)
\]
is a very weak solution of \eqref{e: maineq} and $u_1^-\in\ell^1(V,\mu).$
Then, for every $\beta'>\beta/\sigma$,
 $u\in L^1_{\mathrm{loc}}([0,\infty);X_{\beta'}).$
Consequently, if $u\in L^1_{\mathrm{loc}}([0,\infty);X_{\delta_0})$
for some $\delta_0>0$, then
\begin{equation}\label{eq:all-exponential-weights}
u\in L^1_{\mathrm{loc}}([0,\infty);X_\varepsilon)
\qquad\text{for every }\varepsilon>0.
\end{equation}
\end{lemma}

\begin{proof}
Fix $T>0$ and set
\[
w_\beta(x):=e^{-\beta d(x,x_0)}.
\]
If $x\sim y$, then $|d(y,x_0)-d(x,x_0)|\leq j$, and hence
\[
w_\beta(y)\leq e^{\beta j}w_\beta(x).
\]
It follows from \eqref{e7f}(ii) that
\begin{equation}\label{eq:Delta-wbeta}
|\Delta w_\beta(x)|
\leq C_\beta w_\beta(x)
\qquad\text{for every }x\in V.
\end{equation}

Let $\zeta\in C_c^2([0,\infty))$ be nonnegative, with
$\zeta\equiv1$ on $[0,T]$ and $\operatorname{supp}\zeta\subseteq[0,2T]$. For $k\in\mathbb N$ define
\[
\Phi_k(x,t)
:=
\zeta(t)w_\beta(x)\bone_{B_k(x_0)}(x).
\]
This is an admissible test function in Definition \ref{defn2}. Since
$u\in L^1_{\mathrm{loc}}([0,\infty);X_\beta)$, the same estimates as in
Proposition~\ref{p: ibp} give
\begin{equation}\label{eq:Deltau-weighted}
\int_0^{2T}
\sum_{x\in V}
|\Delta u(x,t)|w_\beta(x)\mu(x)\,dt<\infty.
\end{equation}
Moreover,
\[
\sum_{x\in V}u_0(x)(\Phi_k)_t(x,0)\mu(x)=0
\]
and, since $0<w_\beta\leq1$,
\begin{align*}
-\sum_{x\in B_k(x_0)}
u_1(x)w_\beta(x)\mu(x)
&=\sum_{x\in B_k(x_0)}
u_1^-(x)w_\beta(x)\mu(x)
-\sum_{x\in B_k(x_0)}
u_1^+(x)w_\beta(x)\mu(x)
\\
&\leq
\sum_{x\in V}u_1^-(x)\mu(x).
\end{align*}
Testing \eqref{e: maineq} with $\Phi_k$ and using
\eqref{eq:Deltau-weighted}, we therefore obtain a constant $C_T>0$,
independent of $k$, such that
\[
\int_0^T
\sum_{x\in B_k(x_0)}
v(x,t)|u(x,t)|^\sigma
e^{-\beta d(x,x_0)}
\mu(x)\,dt
\leq C_T.
\]
Letting $k\to\infty$, we get
\begin{equation}\label{eq:weighted-nonlinear-integrability}
\int_0^T
\sum_{x\in V}
v(x,t)|u(x,t)|^\sigma
e^{-\beta d(x,x_0)}
\mu(x)\,dt<\infty.
\end{equation}

Now let $\beta'>\beta/\sigma$ and set
\[
\gamma
:=
\frac{\sigma\beta'-\beta}{\sigma-1}>0.
\]
Choose $R\geq R_0$ sufficiently large that
\[
T\leq2R^{\frac{1+\alpha}{2}}
\qquad\text{and}\qquad
\frac{\delta_1}{R}\leq\gamma.
\]
By H\"older's inequality,
\begin{align*}
&\int_0^T
\sum_{x\in V\setminus B_R(x_0)}
|u(x,t)|e^{-\beta'd(x,x_0)}\mu(x)\,dt\\
&\quad\leq
\left(
\int_0^T
\sum_{x\in V\setminus B_R(x_0)}
v(x,t)|u(x,t)|^\sigma
e^{-\beta d(x,x_0)}
\mu(x)\,dt
\right)^{1/\sigma}\\
&\qquad\quad\times
\left(
\int_0^T
\sum_{x\in V\setminus B_R(x_0)}
v^{-\frac1{\sigma-1}}(x,t)
e^{-\gamma d(x,x_0)}
\mu(x)\,dt
\right)^{\frac{\sigma-1}{\sigma}}.
\end{align*}
The first factor is finite by
\eqref{eq:weighted-nonlinear-integrability}. Since
$\delta_1/R\leq\gamma$,
\[
e^{-\gamma d(x,x_0)}
\leq
e^{-\delta_1\frac{d(x,x_0)}{R}},
\]
and hence the second factor is finite by
\eqref{eq:sign-changing-volume-2}. On the finite set $B_R(x_0)$,
\[
\int_0^T
\sum_{x\in B_R(x_0)}
|u(x,t)|e^{-\beta'd(x,x_0)}\mu(x)\,dt<\infty
\]
follows directly from
$u\in L^1_{\mathrm{loc}}([0,\infty);X_\beta)$. Thus
\[
u\in L^1_{\mathrm{loc}}([0,\infty);X_{\beta'}).
\]

Finally, choose $q\in(1/\sigma,1)$ and set
$\beta_n=q^n\delta_0$. Iterating the first part of the lemma yields
\[
u\in L^1_{\mathrm{loc}}([0,\infty);X_{\beta_n})
\qquad\text{for every }n\in\mathbb N.
\]
Since $\beta_n\to0$, for every $\varepsilon>0$ there exists $n$ such that
$\beta_n<\varepsilon$. Since $X_{\beta_n}\subset X_\varepsilon$, we obtain
\eqref{eq:all-exponential-weights}.
\end{proof}

We are now ready to prove Theorem~\ref{thm:sign-changing}.

\begin{proof}[Proof of Theorem~\ref{thm:sign-changing}]
By Lemma~\ref{lem:bootstrap-Xdelta},
\begin{equation}\label{eq:bootstrap-consequence}
u\in L^1_{\mathrm{loc}}([0,\infty);X_\varepsilon)
\qquad\text{for every }\varepsilon>0.
\end{equation}

Let $\psi:[-j,\infty)\to(0,\infty)$ be such that
\[
\psi\in C^2([-j,\infty)),\qquad
\psi'\leq0,\qquad
\psi\equiv1\quad\text{in }[-j,1],
\]
and
\[
\psi(r)=e^{-\delta_1r}
\qquad\text{for every }r\geq2.
\]
In particular, there exist constants $C_1,C_2>0$ such that
\begin{equation}\label{eq:psi-comparison}
C_2e^{-\delta_1r}
\leq\psi(r)\leq C_1e^{-\delta_1r},
\end{equation}
and
\begin{equation}\label{eq:psi-derivatives}
|\psi'(r)|+|\psi''(r)|
\leq C_1e^{-\delta_1r}
\end{equation}
for every $r\in[-j,\infty)$.

Let $\eta:[0,\infty)\to[0,\infty)$ satisfy
\[
\eta\in C^2([0,\infty)),\qquad
\eta\equiv1\ \text{on }[0,1],\qquad
\eta\equiv0\ \text{on }[2,\infty),
\]
\[
\eta'\leq0,\qquad
|\eta'|+|\eta''|\leq C,\qquad
0\leq\eta\leq1.
\]
For
\[
a:=\frac{1+\alpha}{2},
\]
define
\begin{equation}\label{eq:phiR-sign-changing}
\varphi_R(x,t)
:=
\eta^s\left(\frac{t}{R^a}\right)
\psi\left(\frac{d(x,x_0)-j}{R}\right),
\end{equation}
where
\[
s>\frac{2\sigma}{\sigma-1},
\qquad
R\geq\max\{R_0,2j\}.
\]

Set
\[
Q_R:=V\times[R^a,2R^a].
\]
From the definition of $\varphi_R$ and the properties of $\eta$ and $\psi$
we have
\begin{equation}\label{eq:phiR-t}
|(\varphi_R)_t(x,t)|
\leq
\frac{C}{R^a}
\eta^{s-1}\left(\frac{t}{R^a}\right)
e^{-\delta_1\frac{d(x,x_0)}{R}}
\bone_{Q_R}(x,t),
\end{equation}
and
\begin{equation}\label{eq:phiR-tt}
|(\varphi_R)_{tt}(x,t)|
\leq
\frac{C}{R^{1+\alpha}}
\eta^{s-2}\left(\frac{t}{R^a}\right)
e^{-\delta_1\frac{d(x,x_0)}{R}}
\bone_{Q_R}(x,t).
\end{equation}

We also claim that
\begin{equation}\label{eq:phiR-Laplacian}
|\Delta\varphi_R(x,t)|
\leq
\frac{C}{R^{1+\alpha}}
\eta^s\left(\frac{t}{R^a}\right)
e^{-\delta_1\frac{d(x,x_0)}{R}}
\bone_{V\setminus B_R(x_0)}(x).
\end{equation}
Indeed,
\begin{align*}
\Delta\varphi_R(x,t)
&=
\eta^s\left(\frac{t}{R^a}\right)
\frac1{\mu(x)}
\sum_{y\sim x}\omega_{xy}
\left[
\psi\left(\frac{d(y,x_0)-j}{R}\right)
-
\psi\left(\frac{d(x,x_0)-j}{R}\right)
\right].
\end{align*}
If $x\in B_R(x_0)$, then $d(y,x_0)\leq R+j$ for every $y\sim x$,
and hence both arguments of $\psi$ belong to $[-j,1]$; therefore
$\Delta\varphi_R(x,t)=0$.

Let now $x\in V\setminus B_R(x_0)$. By Taylor's formula,
\begin{align*}
\Delta\varphi_R(x,t)
&=
\frac1R
\eta^s\left(\frac{t}{R^a}\right)
\psi'\left(\frac{d(x,x_0)-j}{R}\right)
\Delta d(x,x_0)\\
&\quad+
\frac1{2R^2}
\eta^s\left(\frac{t}{R^a}\right)
\frac1{\mu(x)}
\sum_{y\sim x}\omega_{xy}
\psi''(\xi_{xy})
\bigl(d(y,x_0)-d(x,x_0)\bigr)^2,
\end{align*}
where $\xi_{xy}$ lies between
$(d(y,x_0)-j)/R$ and $(d(x,x_0)-j)/R$.
Since $R\geq2j$ and $d(x,x_0)\geq R$,
\[
\xi_{xy}\geq\frac{d(x,x_0)-2j}{R}
\geq\frac{d(x,x_0)}{R}-1.
\]
Using \eqref{e7fn}(v$'$), \eqref{eq:psi-derivatives},
$|d(y,x_0)-d(x,x_0)|\leq j$, and \eqref{e7f}(ii), we obtain
\begin{align*}
|\Delta\varphi_R(x,t)|
&\leq
\frac{C}{R\,d^\alpha(x,x_0)}
\eta^s\left(\frac{t}{R^a}\right)
e^{-\delta_1\frac{d(x,x_0)}{R}}
+
\frac{C}{R^2}
\eta^s\left(\frac{t}{R^a}\right)
e^{-\delta_1\frac{d(x,x_0)}{R}}\\
&\leq
\frac{C}{R^{1+\alpha}}
\eta^s\left(\frac{t}{R^a}\right)
e^{-\delta_1\frac{d(x,x_0)}{R}},
\end{align*}
where in the last inequality we used $d(x,x_0)\geq R$ and
$\alpha\in[0,1]$. This proves \eqref{eq:phiR-Laplacian}.

For $k\in\mathbb N$ define
\[
\varphi_{R,k}(x,t)
:=
\varphi_R(x,t)\bone_{B_k(x_0)}(x).
\]
This is an admissible test function in Definition \ref{defn2}. By
\eqref{eq:bootstrap-consequence}, with $\varepsilon=\delta_1/R$, the
estimates in Proposition~\ref{p: ibp} show that all terms
containing $u$ or $\Delta u$ converge absolutely as $k\to\infty$.

Moreover,
\[
(\varphi_R)_t(x,0)=0,
\]
and hence the term involving $u_0$ vanishes. For $k\geq R+j$,
\begin{align}
-\sum_{x\in B_k(x_0)}
u_1(x)\varphi_R(x,0)\mu(x)
&\leq
-\sum_{x\in B_{R+j}(x_0)}u_1^+(x)\mu(x)
+
\sum_{x\in V}u_1^-(x)\mu(x).
\label{eq:initial-term-estimate}
\end{align}
Indeed, $\varphi_R(x,0)=1$ on $B_{R+j}(x_0)$ and
$0<\varphi_R(x,0)\leq1$ on $V$.

Testing \eqref{e: maineq} with $\varphi_{R,k}$, using
\eqref{eq:initial-term-estimate}, and then letting $k\to\infty$, we obtain
\begin{align}
\int_0^\infty\sum_{x\in V}
v(x,t)|u(x,t)|^\sigma\varphi_R(x,t)\mu(x)\,dt
&\leq
-\int_0^\infty\sum_{x\in V}
\Delta u(x,t)\varphi_R(x,t)\mu(x)\,dt
\nonumber\\
&\quad+
\int_0^\infty\sum_{x\in V}
u(x,t)(\varphi_R)_{tt}(x,t)\mu(x)\,dt
\nonumber\\
&\quad-
\sum_{x\in B_{R+j}(x_0)}u_1^+(x)\mu(x)
+
\sum_{x\in V}u_1^-(x)\mu(x).
\label{eq:main-sign-changing}
\end{align}

By Proposition~\ref{p: ibp},
\[
-\int_0^\infty\sum_{x\in V}
\Delta u\,\varphi_R\,\mu \, dt
=
-\int_0^\infty\sum_{x\in V}
u\,\Delta\varphi_R\,\mu\,  dt.
\]
Using \eqref{eq:phiR-Laplacian}, Young's inequality,
\eqref{eq:psi-comparison}, and
\eqref{eq:sign-changing-volume-2}, for every $\varepsilon>0$ we obtain
\begin{align*}
\left|
\int_0^\infty\sum_{x\in V}
u\,\Delta\varphi_R\,\mu \, dt
\right|
&\leq
\varepsilon
\int_0^\infty\sum_{x\in V}
v|u|^\sigma\varphi_R\,\mu \, dt
+C_\varepsilon.
\end{align*}
Similarly, using \eqref{eq:phiR-tt}, Young's inequality,
\eqref{eq:sign-changing-volume-1} and
\eqref{eq:sign-changing-volume-2},
\[
\left|
\int_0^\infty\sum_{x\in V}
u(\varphi_R)_{tt}\,\mu\,  dt
\right|
\leq
\varepsilon
\int_0^\infty\sum_{x\in V}
v|u|^\sigma\varphi_R\,\mu\,  dt
+C_\varepsilon.
\]
Choosing $\varepsilon>0$ sufficiently small in
\eqref{eq:main-sign-changing}, we obtain
\begin{align}
\int_0^\infty\sum_{x\in V}
v(x,t)|u(x,t)|^\sigma\varphi_R(x,t)\mu(x)\,dt
\leq
C
-\sum_{x\in B_{R+j}(x_0)}u_1^+(x)\mu(x)
+\sum_{x\in V}u_1^-(x)\mu(x).
\label{eq:first-global-estimate}
\end{align}
Since $\varphi_R\equiv1$ on
$B_R(x_0)\times[0,R^a]$, it follows that
\begin{align}
\int_0^{R^a}\sum_{x\in B_R(x_0)}
v(x,t)|u(x,t)|^\sigma\mu(x)\,dt
\leq
C
-\sum_{x\in B_{R+j}(x_0)}u_1^+(x)\mu(x)
+\sum_{x\in V}u_1^-(x)\mu(x).
\label{eq:local-global-estimate}
\end{align}

Set
\[
P:=\sum_{x\in V}u_1^+(x)\mu(x)\in[0,+\infty],
\qquad
N:=\sum_{x\in V}u_1^-(x)\mu(x)<\infty.
\]
If $P=+\infty$, then the right-hand side of
\eqref{eq:local-global-estimate} becomes negative for $R$ sufficiently
large, which is impossible. Hence the existence of a very weak solution
already implies $P<\infty$. Letting $R\to\infty$ in
\eqref{eq:local-global-estimate}, we therefore obtain
\begin{equation}\label{eq:global-Lsigma}
\int_0^\infty\sum_{x\in V}
v(x,t)|u(x,t)|^\sigma\mu(x)\,dt<\infty.
\end{equation}

We now return to \eqref{eq:main-sign-changing} and estimate the two integral
terms by H\"older's inequality rather than Young's inequality. From
\eqref{eq:phiR-Laplacian} and
\eqref{eq:sign-changing-volume-2} we obtain
\begin{align*}
\left|
\int_0^\infty\sum_{x\in V}
u\,\Delta\varphi_R\,\mu \, dt
\right|
&\leq
C
\left(
\int_0^\infty
\sum_{x\in V\setminus B_R(x_0)}
v(x,t)|u(x,t)|^\sigma\mu(x)\,dt
\right)^{1/\sigma}.
\end{align*}
Likewise, by \eqref{eq:phiR-tt},
\eqref{eq:sign-changing-volume-1} and
\eqref{eq:sign-changing-volume-2},
\begin{align*}
\left|
\int_0^\infty\sum_{x\in V}
u(\varphi_R)_{tt}\, \mu \, dt
\right|
&\leq
C
\left(
\int_{R^a}^{2R^a}
\sum_{x\in V}
v(x,t)|u(x,t)|^\sigma\mu(x)\,dt
\right)^{1/\sigma}.
\end{align*}
Therefore,
\begin{align*}
\int_0^{R^a}\sum_{x\in B_R(x_0)}
v(x,t)|u(x,t)|^\sigma\mu(x)\,dt
&\leq
C
\left(
\int_0^\infty
\sum_{x\in V\setminus B_R(x_0)}
v|u|^\sigma\,\mu \, dt
\right)^{1/\sigma}\\
&\quad+
C
\left(
\int_{R^a}^{2R^a}
\sum_{x\in V}
v|u|^\sigma\,\mu \, dt
\right)^{1/\sigma}\\
&\quad-
\sum_{x\in B_{R+j}(x_0)}u_1^+(x)\mu(x)
+
N.
\end{align*}
By \eqref{eq:global-Lsigma}, the first two terms on the right-hand side tend
to zero as $R\to\infty$. Hence, using
\eqref{eq:initial-mass-sign},
\[
\int_0^\infty\sum_{x\in V}
v(x,t)|u(x,t)|^\sigma\mu(x)\,dt
\leq -P+N\leq0.
\]
Since the integrand is nonnegative, we conclude that $u\equiv0$.
\end{proof}

\begin{proof}[Proof of Corollary~\ref{c: nosegno}]
We apply Theorem~\ref{thm:sign-changing}. For every $R\geq R_0$,
using $v(x,t)\geq g(x)$ and \eqref{eq:sign-changing-g}, we have
\begin{align*}
&\int_{R^{\frac{1+\alpha}{2}}}^{2R^{\frac{1+\alpha}{2}}}
\sum_{x\in B_R(x_0)}
v^{-\frac1{\sigma-1}}(x,t)
e^{-\delta_1\frac{d(x,x_0)}{R}}
\mu(x)\,dt\\
&\qquad\leq
R^{\frac{1+\alpha}{2}}
\sum_{x\in V}
g^{-\frac1{\sigma-1}}(x)
e^{-\delta_1\frac{d(x,x_0)}{R}}
\mu(x)\\
&\qquad\leq
CR^{\frac{1+\alpha}{2}
+\frac{(\alpha+1)(\sigma+1)}{2(\sigma-1)}}
=
CR^{\frac{(\alpha+1)\sigma}{\sigma-1}}.
\end{align*}
Similarly,
\begin{align*}
&\int_0^{2R^{\frac{1+\alpha}{2}}}
\sum_{x\in V\setminus B_R(x_0)}
v^{-\frac1{\sigma-1}}(x,t)
e^{-\delta_1\frac{d(x,x_0)}{R}}
\mu(x)\,dt\\
&\qquad\leq
2R^{\frac{1+\alpha}{2}}
\sum_{x\in V}
g^{-\frac1{\sigma-1}}(x)
e^{-\delta_1\frac{d(x,x_0)}{R}}
\mu(x)\leq
CR^{\frac{(\alpha+1)\sigma}{\sigma-1}}.
\end{align*}
Thus \eqref{eq:sign-changing-volume-1} and
\eqref{eq:sign-changing-volume-2} are satisfied, and the conclusion follows
from Theorem~\ref{thm:sign-changing}.
\end{proof}

\begin{proof}[Proof of Corollary~\ref{c: nosegnoV=1}]
We apply Corollary~\ref{c: nosegno} with $g\equiv1$.
Fix any $\delta_1>0$. If $k_0>R_0$ is sufficiently large, then for every
$R\geq k_0$,
\begin{align*}
\sum_{x\in V}e^{-\delta_1\frac{d(x,x_0)}{R}}\mu(x)
&\leq
\sum_{k=k_0}^\infty
e^{-\delta_1\frac{k}{R}}
\operatorname{Vol}
\bigl(B_{k+1}(x_0)\setminus B_k(x_0)\bigr)
+
\operatorname{Vol}(B_{k_0}(x_0))\\
&\leq
C\sum_{k=k_0}^\infty
e^{-\delta_1\frac{k}{R}}
k^{\frac{(\alpha+1)(\sigma+1)}{2(\sigma-1)}-1}
+C\\
&\leq
C\int_{R_0}^\infty
e^{-\delta_1\frac{r}{R}}
r^{\frac{(\alpha+1)(\sigma+1)}{2(\sigma-1)}-1}\,dr
+C\leq
CR^{\frac{(\alpha+1)(\sigma+1)}{2(\sigma-1)}}.
\end{align*}
Thus condition \eqref{eq:sign-changing-g} is satisfied, and the assertion
follows from Corollary~\ref{c: nosegno}.
\end{proof}

\section{Examples and comments}\label{ex}

\begin{example}\label{ex: Zn}
    Let $(\Z^N, \omega, \mu)$ be the $N$-dimensional integer lattice graph endowed with the Euclidean distance $d(x,y)=|x-y|$ for every $x,y\in\mathbb{Z}^N$ (see e.g. \cite[Example 5.1]{MPS1} for more information on this class of graphs). Here $\omega_{xy}=1$ for every pair of neighboring points $x,y$ and $\mu(x)=2N$ for every $x\in\mathbb{Z}^N$. We have $\Vol(B_R)=C_NR^N + O(R^{\theta})$ as $R$ tends to $+\infty$,
\begin{itemize}
    \item for every $\theta>N-2$ if $N\geq 4$ (see \cite{Fri82});
    \item for every $\theta> \frac{29}{22}$ if $N=3$ (see \cite{ChaIw95});
    \item for every $\theta>\frac{46}{73}$ if $N=2$ (see \cite{Hux93}).
    \end{itemize}
    Moreover $\Vol(B_R)=2(2[R]+1)$ if $N=1$. Therefore, it holds that $\operatorname{Vol}(B_R\setminus B_{R-1})\leq C R^{N-1}$ and $\operatorname{Vol}(B_R)\leq C R^{N}$. We also note that
    \begin{equation*}
        \Delta d^2(x,x_0)=\frac{1}{2N}\sum_{y\sim x} (d^2(y,x_0)-d^2(x,x_0))=1,
    \end{equation*}
    hence we have $|\Delta d^2(x,x_0)|=1$. Therefore, by Remark \ref{r: carattmodulolaplaciano}, it holds that
    \begin{equation*}
        |\Delta d(x,x_0)|\leq \frac{C}{d(x,x_0)},
    \end{equation*}
    for every $x\in V$ such that $d(x,x_0)>0$, and thus conditions \eqref{e7f} and \eqref{e7fn} hold with $\alpha=1$. Now let $u\in L^1_{\textup{loc}}([0,\infty), X_{\delta_0})$ be a very weak solution of \eqref{e: maineq} with $v\equiv1$, for some $\delta_0>0$. Suppose that $u_1^-\in \ell^1(V,\mu)$ and
    \begin{equation*}
        \sum_{x\in V}u_1(x)\mu(x)\geq 0.
    \end{equation*}
    We observe that if
    \begin{equation}\label{4998}
    N=1, \, \sigma>1 \qquad \textrm{or}\qquad N\geq 2,\, 1<\sigma\leq \frac{N+1}{N-1},
    \end{equation}
    then it holds that $\operatorname{Vol}(B_R\setminus B_{R-1})\leq CR^{N-1}\leq C R^{\frac{\sigma+1}{\sigma-1}-1}$.
    Thus $u\equiv 0$ by Corollary \ref{c: nosegnoV=1}. On the other hand, suppose $u$ is a nonnegative very weak solution of \eqref{e: maineq} with $v\equiv1$ and assume \eqref{4998} holds. Then $\operatorname{Vol}(B_R)\leq CR^{N}\leq C R^{\frac{\sigma+1}{\sigma-1}}$. If
$$\liminf_{R\to\infty}\left\{\sum_{x\in B_R}u^+_1(x) - \sum_{x\in B_{2R}}u_1^-(x)\right\}\geq 0,$$
    then $u\equiv 0$ by Corollary \ref{c: segno2}.

    This example should be compared with the results in $\R^N$ (see for example \cite{Kato80}, \cite{MP01}, \cite{Sha85}) and with those on Riemannian manifolds \cite{MPS20}.
\end{example}

\begin{example}
    Let $\mathbb{T}_N$ be a homogeneous tree of degree $N+1$, with $N>1$, and let $x_0$ denote the root of $\mathbb{T}_N$. We consider $\mathbb{T}_N$ endowed with the natural graph distance $d$ and $\mu(x)=1$ for every $x\in\mathbb{T}_N$. We recall that $d(x,y)$ is equal to the cardinality of the set of edges of the unique minimal path between $x$ and $y$ (we refer e.g. to \cite{MPS1} for more information on trees). Then conditions \eqref{e7f} and \eqref{e7fn} hold with $\alpha=0$. Let $v\colon \mathbb{T}_N\times[0,\infty)\to \R$ be such that
    \begin{equation*}
        v(x,t)\geq g(x)= C d(x,x_0)^{\frac{\sigma-3}{2}}N^{(\sigma-1)d(x,x_0)},
    \end{equation*}
    for every $x\in \mathbb{T}_N$ such that $d(x,x_0)\geq R_0 > 1$, with $\sigma>1$, and $$v(x,t)\geq g(x)=C$$ for every $x\in \mathbb{T}_N$ such that $d(x,x_0)< R_0$.
    We observe that, since the tree $\mathbb{T}_N$ is homogeneous, for each $k  \in \N$ with $k\geq1$ we have
    \begin{equation*}
        \mu(\{x\in \mathbb{T}_N\colon d(x,x_0)=k\})=|\{x\in \mathbb{T}_N\colon d(x,x_0)=k\}|=(N+1)N^{k-1}.
    \end{equation*}
    From this we get, for every $R\geq R_0$ and $\delta>0$,
    \begin{equation*}
    \begin{split}
        \sum_{x\in V} g^{-\frac{1}{\sigma-1}}(x)e^{-\delta\frac{d(x,x_0)}{R}}\mu(x)&\leq
        C\sum_{x\in V\setminus\{x_0\}} d(x,x_0)^{\frac{3-\sigma}{2(\sigma-1)}}N^{-d(x,x_0)}e^{-\delta\frac{d(x,x_0)}{R}}+C\\
        &=C\sum_{k=1}^{\infty}\sum_{\{x\in \mathbb{T}_N\colon d(x,x_0)=k\}}k^{\frac{3-\sigma}{2(\sigma-1)}}N^{-k}e^{-\delta\frac{k}{R}}+C\\
        &\leq C\sum_{k=1}^{\infty}k^{\frac{3-\sigma}{2(\sigma-1)}}e^{-\delta\frac{k}{R}}+C\\
        &\leq C\int_1^{\infty} r^{\frac{3-\sigma}{2(\sigma-1)}}e^{-\delta\frac{r}{R}} \, dr+C\leq CR^{\frac{\sigma+1}{2(\sigma-1)}}.
        \end{split}
    \end{equation*}
    Therefore, if $\delta_0>0$, $u_1^-\in \ell^1(V,\mu)$, $\sum_{x\in V}u_1(x)\mu(x)\geq 0$ and $u\in L^1_{\textup{loc}}([0,\infty),X_{\delta_0})$ is a very weak solution of \eqref{e: maineq}, then by Corollary \ref{c: nosegno} we have  $u\equiv 0$. Similarly, we have
    \begin{equation*}
    \begin{split}
    \sum_{x\in B_R(x_0)} g^{-\frac{1}{\sigma-1}}(x)\mu(x)&\leq C\sum_{k=1}^{[R]+1}\sum_{\{x\in \mathbb{T}_N\colon d(x,x_0)=k\}} k^{\frac{3-\sigma}{2(\sigma-1)}} N^{-k}+C\\
    &=C\sum_{k=1}^{[R]+1} k^{\frac{3-\sigma}{2(\sigma-1)}}+C\leq CR^\frac{\sigma+1}{2(\sigma-1)}.
    \end{split}
    \end{equation*}
    Thus if $u$ is a nonnegative very weak solution of \eqref{e: maineq} and
    $$\liminf_{R\to\infty}\left\{\sum_{x\in B_R(x_0)}u^+_1(x) - \sum_{x\in B_{2R}(x_0)}u_1^-(x)\right\}\geq 0,$$
    by Corollary \ref{c: segno} we have  $u\equiv 0$.
\end{example}

\begin{example}
    Let $N\in \N$ and consider $V=\Z^N\times W$, where $(W,\omega_W,\mu_W)$ is any finite weighted graph. For $(x_1,w_1),(x_2,w_2)\in \Z^N\times W$ we set $(x_1,w_1)\sim(x_2,w_2)$ if and only if $x_1=x_2$ and $w_1\sim w_2$ or $x_1\sim x_2$ and $w_1= w_2$. Therefore we can define $\omega_V\colon V\times V\to [0,\infty)$ and $\mu_V\colon V\to (0,\infty)$ by
    \begin{equation}
        \omega_V((x_1,w_1),(x_2,w_2))=\begin{cases}
1 \,\,\,\, &\text{if} \,\,x_1\sim x_2,\, w_1=w_2\\
\omega_W(w_1,w_2)\,\,\,\, &\text{if}\,\, x_1= x_2,\, w_1\sim w_2,\\
0&\text{otherwise}
        \end{cases}
    \end{equation}
and
    \begin{equation}
\mu_V(x,w)=\max\{2N, \mu_W(w)\},
    \end{equation}
    in particular there are two positive constants $c_1,c_2$ such that $c_1\leq \mu_V(x,w)\leq c_2$ for every $(x,w)\in V$. Observe that for a function $u\colon V\to\R$, we have
    \begin{equation*}
    \begin{split}
        \Delta_V u(x,w)&=\frac{1}{\mu_V(x,w)}\sum_{y\sim x}(u(y,w)-u(x,w))+\frac{1}{\mu_V(x,w)}\sum_{z\sim w} \omega_W(w,z)(u(x,z)-u(x,w))\\
        &=\frac{2N}{\mu_V(x,w)}(\Delta_{\Z^N}u(\cdot,w))(x) + \frac{\mu_W(w)}{\mu_V(x,w)}(\Delta_W u(x,\cdot))(w).
        \end{split}
    \end{equation*}
    We endow the weighted graph $(V,\omega_V,\mu_V)$ with the distance
    \begin{equation*}
        d_V((x,w),(y,z))=\sqrt{|x-y|^2+d_W^2(w,z)},
    \end{equation*}
    where $|\cdot|$ is the Euclidean distance restricted to $\Z^N$, see also Example \ref{ex: Zn}, and where $d_W$ is the graph distance on $W$. Let $w_0\in W$ be fixed. Since $W$ is finite, the distance $d_W(\cdot,w_0)$ satisfies \eqref{e7fn}-(v') for each $\alpha\in [0,1]$. The same holds for $d_V$, indeed
    \begin{equation*}
        \begin{split}
            &|\Delta d^2_V((x,w),(0,w_0))|\\
            &\quad=\left|\frac{1}{\mu_V(x,w)}\sum_{z\sim w}\omega_W(z,w) [d^2_W(z,w_0)-d^2_W(w,w_0)] +\frac{1}{\mu_V(x,w)} \sum_{y\sim x}[|y|^2-|x|^2]\right|\\
            &\quad\leq |\Delta d_W^2(w,w_0)|+|\Delta|x|^2|\leq C.
        \end{split}
    \end{equation*}
    Then, by Remark \ref{r: carattmodulolaplaciano}, $d_V$ satisfies \eqref{e7f} and \eqref{e7fn} with $\alpha=1$.
    We denote by $D$ the diameter of $W$. For $R>1$ sufficiently large and $(x,w)\in B_R^V((0,w_0))\setminus B_{R-1}^{V}(0,w_0)$, where $B_R^V((0,w_0))$ denotes the ball of center $(0,w_0)$ and radius $R$ in $(V,d_V)$. We have
    \begin{equation*}
        R-1\leq d_V((x,w),(0,w_0))\leq R
    \end{equation*}
    which implies
    \begin{equation*}
        (R-1)^2\leq |x|^2 + d_W^2(w,w_0)\leq R^2,
    \end{equation*}
    from which we get
    \begin{equation*}
        (R-2)^2\leq (R-1)^2 -D^2\leq |x|^2\leq R^2.
    \end{equation*}
    This implies that $x\in B_R^{\Z^N}(0)\setminus B_{R-2}^{\Z^N}(0)$. Hence we get $B_R^V((0,w_0))\setminus B_{R-1}^{V}((0,w_0))\subset (B_R^{\Z^N}(0)\setminus B_{R-2}^{\Z^N}(0))\times W$. Then, we can deduce the following estimate on the volume growth of the balls
    \begin{equation*}
    \begin{split}
        \mu_V(B_R^V((0,w_0))\setminus B_{R-1}^{V}((0,w_0)))&\leq \mu_V( (B_R^{\Z^N}(0)\setminus B_{R-2}^{\Z^N}(0))\times W)\\
        &\leq C| (B_R^{\Z^N}(0)\setminus B_{R-2}^{\Z^N}(0))\times W|\\
        &\leq C| (B_R^{\Z^N}(0)\setminus B_{R-2}^{\Z^N}(0))|\leq C R^{N-1},
        \end{split}
    \end{equation*}
    where we have used the same estimate as in Example \ref{ex: Zn}.
         Therefore, if $v\equiv 1$, $\delta_0>0$, $u_1^-\in \ell^1(V,\mu_V)$, $\sum_{V}u_1\mu_V\geq 0$, and $u\in L^1_{\textup{loc}}([0,\infty),X_{\delta_0})$ is a very weak solution of \eqref{e: maineq}, if $1<\sigma\leq \frac{N+1}{N-1}$ and $N\geq 2$ or if $\sigma>1$ and $N=1$ by Corollary \ref{c: nosegnoV=1} we have  $u\equiv 0$.

Moreover for $R>1$ sufficiently large we have
\begin{equation*}
            \mu_V(B_R^V((0,w_0)))\leq \mu_V( B_R^{\Z^N}(0)\times W)\leq C|B_R^{\Z^N}(0)|\leq CR^N,
\end{equation*}
Hence if $u$ is a nonnegative very weak solution of \eqref{e: maineq} with $v\equiv 1$ and
 $$\liminf_{R\to\infty}\left\{\sum_{(x,w)\in B^V_R((0,w_0))}u^+_1(x,w)\mu_V(x,w) - \sum_{(x,w)\in B^V_{2R}((0,w_0))}u_1^-(x,w)\mu_V(x,w)\right\}\geq 0,$$
if $1<\sigma\leq \frac{N+1}{N-1}$ and $N\geq 2$ or if $\sigma>1$ and $N=1$ by Corollary \ref{c: segno2} we have  $u\equiv 0$.
\end{example}

\section{Finite graphs}\label{finite}

In this section we investigate problem \eqref{e: maineq}, when the weighted graph $(V,\omega,\mu)$ is finite. Adapting the proof of \cite[Theorem 5.2]{MPS2} to our context we get the following result. For the sake of completeness we include a sketch of the proof.

\begin{theorem}\label{teo3}
    Let $(V,\omega,\mu)$ be a weighted \emph{finite} graph. Let $v\colon V\times [0,\infty)\to \R$ be a positive function, $\sigma>1$. Suppose that for every $R\geq R_0>0$
	\begin{equation}\label{e: stimavolumi2}
		\int_R^{2R}\sum_{x\in V} v^{-\frac{1}{\sigma-1}}(x,t)\mu(x)\,dt\leq CR^{\frac{2\sigma}{\sigma-1}}.
	\end{equation}
Let $u\colon V\times [0,\infty)\to \R$ be a very weak solution of \eqref{e: maineq} and suppose that
\begin{equation*}
    \sum_{x\in V}u_1(x)\mu(x)\geq 0,
\end{equation*}
then $u\equiv 0$.
\end{theorem}

\begin{proof}
    Let $\varphi\colon [0,\infty)\to \R$ be a function in $C^2([0,\infty))$ such that $\varphi\equiv 1$ in $[0,1]$, $\varphi\equiv 0$ in $[2,\infty)$, $\varphi'(t)\leq 0$ for each $t\geq 0$. Since $V$ is a finite set and $u$ is a very weak solution of \eqref{e: maineq}, we have that $\phi^s_R(t)=\varphi^s(\frac{t}{R})$, with $s\geq \frac{2\sigma}{\sigma-1}$ satisfies
 \begin{equation}\label{e: ineq to estimate finite}
\begin{split}
\int_0^{\infty}\sum_{x\in V} \mu(x)v(x,t)|u(x,t)|^\sigma\phi_R^s(t)\, dt
&\leq - \int_0^{\infty}\sum_{x\in V}\Delta u(x,t)\phi_R^s(t)\mu(x)\, dt\\
&\,\,\,\,\,\,\,\,+\sum_{x\in V}u_0(x)(\phi_R^s)_t(x,0)\mu(x)-\sum_{x\in V}u_1(x)\phi_R^s(0)\mu(x)\\
&\,\,\,\,\,\,\,\, +\int_{0}^{\infty}\sum_{x\in V}u(x,t)(\phi_R^s)_{tt}(t)\mu(x) \, dt.
\end{split}
\end{equation}
We observe that
\begin{enumerate}[(i)]
    \item $|(\phi_R^s)_{tt}(t)|\leq \frac{C}{R^2}\phi_R^{s-2}(t)\bone_{[R,2R]}(t)$;
    \item $\phi_R^s(0)=1$, and $(\phi_R^s)_t(0)=0$;
    \item \[
    \int_0^\infty \sum_{x\in V} \Delta u(x,t)\phi_R^s(t)\mu(x)\, dt= \int_0^\infty \sum_{x\in V} u(x,t)\Delta\phi_R^s(t)\mu(x)=0.
    \]
\end{enumerate}
From \eqref{e: ineq to estimate finite} it follows that
\begin{equation}\label{e52}
    \int_0^{\infty}\sum_{x\in V} \mu(x)v(x,t)|u(x,t)|^\sigma\phi_R^s(t)\, dt \leq \frac{C}{R^2}\int_{R}^{2R} \sum_{x\in V}|u(x,t)| \phi^{s-2}_R(t)\mu(x)\, dt.
\end{equation}

Applying Young's inequality we get

\begin{equation*}
    \begin{split}
       \int_0^{\infty}\sum_{x\in V} \mu(x)v(x,t)|u(x,t)|^\sigma\phi_R^s(t)\, dt &\leq \frac{C}{R^\frac{2\sigma}{\sigma-1}}\int_R^{2R}\sum_{x\in V} v^{-\frac{1}{\sigma-1}}(x,t)\phi_R^{s-\frac{2\sigma}{\sigma-1}}\mu(x)\, dt \\
       &\quad +\frac{1}{\sigma}\int_R^{2R}\sum_{x\in V}|u(x,t)|^{\sigma}v(x,t)\phi_R^s(t)\mu(x)\,dt,
    \end{split}
\end{equation*}
which implies
\begin{equation*}
    \int_0^R\sum_{x\in V}v(x,t)|u(x,t)|^{\sigma}\mu(x)\,dt\leq\frac{C}{R^{\frac{2\sigma}{\sigma-1}}}\int_R^{2R}\sum_{x\in V}v^{-\frac{1}{\sigma-1}}(x,t)\mu(x)\,dt\leq C.
\end{equation*}
Then, passing to the limit as $R$ tends to $+\infty$, we obtain
\begin{equation*}
    \int_0^\infty \sum_{x\in V} v(x,t)|u(x,t)|^\sigma\mu(x)\,dt\leq C.
\end{equation*}
Finally, by a standard application of H\"older's inequality starting from \eqref{e52}, we get
\begin{equation*}
    \int_0^{\infty}\sum_{x\in V} v(x,t) |u(x,t)|^\sigma\mu(x) \, dt=0,
\end{equation*}
which shows that $u\equiv 0$.
\end{proof}

As a consequence we get the following corollary. Since the proof is immediate, we will omit it.

\begin{corollary}\label{coroll4}
    Let $(V,\omega,\mu)$ be a weighted \emph{finite} graph. Let $v\colon V\to \R$ be a positive function, $\sigma>1$.
Let $u\colon V\times [0,\infty)\to \R$ be a  very weak solution of \eqref{e: maineq} and suppose that
\begin{equation*}
    \sum_{x\in V}u_1(x)\mu(x)\geq 0,
\end{equation*}
then $u\equiv 0$.
\end{corollary}



\begin{thebibliography}{1000}

\bibitem{AZ} S. Avdonin, Y. Zhao, {\it Exact controllability of the wave equation on graphs}, Appl. Math. Optim. {\bf 85} (2022), Paper No. 1, 44 pp.

\bibitem{AN} S. Avdonin, S. Nicaise, {\it Source identification problems for the wave equation on graphs}, Inverse Problems  {\bf 31} (2015), 095007.

\bibitem{BCG}  M. Barlow, T. Coulhon, A. Grigor’yan, {\it Manifolds and graphs with slow heat kernel decay}, Invent. Math. {\bf 144} (2001), 609--649.

\bibitem{BP1} S. Biagi, F. Punzo, {\it Phragmèn-Lindel\"of type theorems for elliptic equations on infinite graphs}, Potential Anal. {\bf 64}, 19 (2026).

\bibitem{BMP1} S. Biagi, G. Meglioli, F. Punzo, {\it A Liouville theorem for elliptic equations with a potential on infinite graphs}, Calc. Var. Partial Differ. Equ. {\bf 63}, 165 (2024).

\bibitem{BMP2} S. Biagi, G. Meglioli, F. Punzo, {\it Phragmèn-Lindel\"of type theorems for parabolic equations on infinite graphs}, preprint (2025) arXiv:2504.08407.

\bibitem{ChaIw95} F. Chamizo, H. Iwaniec, {\it On the sphere problem,} Rev. Mat. Iberoamericana \textbf{11}, (1995), 417--429.

\bibitem{DvC} D.-E. von Criegern, {\it Nonexistence results for a general class of parabolic problems with a potential on weighted graphs}, Nonlinear Differ. Equ. Appl. \textbf{33}, 44 (2026).

\bibitem{Fri82} F. Fricker, Einführung in die Gitterpunktlehre, Birkhäuser Verlag, 1982.

\bibitem{FT}  J. Friedman, J.-P. Tillich, {\it Wave equations for graphs and the edge-based Laplacian}, Pacific J. Math., {\bf 216} (2004), 229--266.

\bibitem{Grig1} A. Grigor’yan, Introduction to Analysis on Graphs, AMS University Lecture Series 71 (2018).

\bibitem{GLY} A. Grigor’yan, Y. Lin, Y. Yang, {\it Kazdan-Warner equation on graph}, Calc. Var. Partial Differ. Equ. {\bf 55}, 92 (2016), 13 pp.

\bibitem{GLY2}  A. Grigor’yan, Y. Lin, Y. Yang, {\it Yamabe type equations on graphs}, J. Diff. Eq. {\bf 261} (2016), 4924--4943.

\bibitem{GT}  A. Grigor’yan, A. Telcs, {\it Sub-Gaussian estimates of heat kernels on infinite graphs}, Duke Math. J. {\bf 109} (2001), 451--510.


\bibitem{GMP} G. Grillo, G. Meglioli, F. Punzo, {\it Blow-up and global existence for semilinear parabolic equations on infinite graphs}, Calc. Var. Partial Differ. Equ. {\bf 65}, 114 (2026).

\bibitem{GHS} Q. Gu, X. Huang, Y. Sun, {\it Semi-linear elliptic inequalities on weighted graphs}, Calc. Var. Partial Differ. Equ. {\bf 62}, 42 (2023).

\bibitem{GSXX} Q. Gu, Y. Sun, J. Xiao, F. Xu, {\it Global positive solution to a semi-linear parabolic equation with potential on Riemannian manifold}, Calc. Var. Partial Differ. Equ. {\bf 59}, 170 (2020).


\bibitem{HaHua} F. Han, B. Hua, {\it Uniqueness class of solutions to a class of linear evolution equations}, J. Diff. Eq. {\bf 409} (2024), 441--460.

\bibitem{HL} B. Hua, Y. Lin,  {\it Stochastic completeness for graphs with curvature dimension conditions}, Adv. Math. {\bf 306} (2017), 279--302.

\bibitem{HW} B. Hua, L. Wang, {\it Dirichlet $p$-Laplacian eigenvalues and Cheeger constants on symmetric graphs}, Adv. Math. {\bf 364} (2020), 106997.

\bibitem{H} X. Huang, {\it On uniqueness class for a heat equation on graphs}, J. Math. Anal. Appl. {\bf 393} (2012), 377--388.

\bibitem{HKS} X. Huang, M. Keller, M. Schmidt, {\it On the uniqueness class, stochastic completeness and volume growth for graphs}, Trans. Amer. Math. Soc. {\bf 373} (2020), 8861--8884.

\bibitem{Hux93} M. N. Huxley, {\it Exponential sums and lattice points. II,} Proc. London Math. Soc. \textbf{66}
(1993), 279--301.

\bibitem{Kato80} T. Kato, {\it Blow-up of solutions of some nonlinear hyperbolic equations}, Comm. Pure Appl. Math. \textbf{33} (1980), 501--505.

\bibitem{KLW} M. Keller, D. Lenz, R.K. Wojciechowski, Graphs and Discrete Dirichlet Spaces, Springer (2021).

\bibitem{KR} M. Keller, C. Rose, {\it Anchored heat kernel upper bounds on graphs with unbounded geometry and anti-trees,} Calc. Var. Partial Differ. Equ. {\bf 63}, 20 (2024).

\bibitem{LX} Y. Lin, Y. Xie, {\it The existence of the solution of the wave equation on graphs}, preprint (2019), arXiv:1908.02137.

\bibitem{LW} Y. Lin, Y. Wu, {\it The existence and nonexistence of global solutions for a semilinear heat equation on graphs}, Calc. Var. Partial Differ. Equ. {\bf 56}, 102 (2017), 22 pp.

\bibitem{M} G. Meglioli, {\it On the uniqueness for the heat equation with density on infinite graphs}, J. Diff. Eq. {\bf 425} (2025), 728--762.

\bibitem{MP} G. Meglioli, F. Punzo, {\it Uniqueness in weighted $\ell^p$ spaces for the Schr\"odinger equation on infinite graphs}, Proc. Amer. Math. Soc. 153 (2025), 1519--1537.

\bibitem{MP2} G. Meglioli, F. Punzo, {\it Uniqueness of solutions to elliptic and parabolic equations on metric graphs,} preprint (2025) arXiv:2503.02551.

\bibitem{MP01} E. Mitidieri and S. Pohozaev, {\it Nonexistence of weak solutions for some degenerate  elliptic and parabolic problems on $\R^n$}, J. Evol. Equ. \textbf{1} (2001), 189--220.

\bibitem{MPS1} D. D. Monticelli, F. Punzo, J. Somaglia, {\it Nonexistence results for semilinear elliptic equations on weighted graphs,} Math. Ann. \textbf{393} (2025), 3395--3418.

\bibitem{MPS2} D. D. Monticelli, F. Punzo, J. Somaglia, {\it Nonexistence of solutions to parabolic problems with a potential on weighted graphs,} J. Diff. Eq. {\bf 453} (2026), 113782.

\bibitem{MPS20} D. D. Monticelli, F. Punzo, M. Squassina, {\it Nonexistence for hyperbolic problems on Riemannian manifolds,}  Asymptot. Anal. \textbf{120} (2020), 87--101.

\bibitem{Mu1} D. Mugnolo, Semigroup Methods for Evolution Equations on Networks, Springer (2014).

\bibitem{PT1} F. Punzo, A. Tesei, {\it Monotonicity results for semilinear parabolic equations on metric graphs}, preprint (2025) arXiv:2502.08361.

\bibitem{PT2} F. Punzo, A. Tesei, {\it Extinction and propagation phenomena for semilinear parabolic equations on metric trees}, preprint (2025) arXiv:2505.10712.

\bibitem{Ru}  Q. Ru, {\it A nonexistence result for a nonlinear wave equation with damping on a Riemannian manifold,} Bound. Value Probl. {\bf 198} (2016).

\bibitem{Sha85} J. Schaeffer, {\it The equation $u_{tt}-\Delta u=|u|^p$ for the critical value of p}, Proc. Royal. Soc. Edinburgh \textbf{101} (1985), 31--44.

\bibitem{Wu} Y. Wu, {\it On nonexistence of global solutions for a semilinear heat equation on graphs}, Nonlinear Anal. {\bf 171} (2018), 73--84.

\end{thebibliography}
\end{document}